\documentclass[leqno]{amsart}

\usepackage{amssymb}

\usepackage{enumitem}

\usepackage{graphicx}

\makeatletter
\@namedef{subjclassname@2020}{%
  \textup{2020} Mathematics Subject Classification}
\makeatother

\usepackage[T1]{fontenc}

\theoremstyle{definition}

\numberwithin{equation}{section}

\usepackage{graphicx, amsmath, amssymb, amsthm, hyperref, verbatim, multicol, mathrsfs, tikz, caption, subcaption}
\usepackage{stmaryrd}
\usepackage{mathtools}  
\usepackage{cleveref}   
\usepackage{todonotes}
\usepackage{tikz-cd}
\usepackage{physics}
\usepackage{enumitem}   
\usepackage[margin=1.25in]{geometry}
\usepackage{comment}

\usepackage{subcaption}
\newtheorem{thm}{Theorem}[section]
\newtheorem{lemm}[thm]{Lemma}
\newtheorem{cor}[thm]{Corollary}

\newtheorem{prop}[thm]{Proposition}

\newtheorem*{conv}{Convention}

\theoremstyle{definition}
\newtheorem{defi}[thm]{Definition}

\newtheorem{rem}[thm]{Remark}

\theoremstyle{remark}

\makeatletter
\newcommand*\rel@kern[1]{\kern#1\dimexpr\macc@kerna}
\newcommand*\widebar[1]{%
  \begingroup
  \def\mathaccent##1##2{%
    \rel@kern{0.8}%
    \overline{\rel@kern{-0.8}\macc@nucleus\rel@kern{0.2}}%
    \rel@kern{-0.2}%
  }%
  \macc@depth\@ne
  \let\math@bgroup\@empty \let\math@egroup\macc@set@skewchar
  \mathsurround\z@ \frozen@everymath{\mathgroup\macc@group\relax}%
  \macc@set@skewchar\relax
  \let\mathaccentV\macc@nested@a
  \macc@nested@a\relax111{#1}%
  \endgroup
}
\makeatother

\DeclareMathOperator{\Euc}{Euc}

\DeclareMathOperator{\diam}{diam}

\DeclareMathOperator{\Area}{Area}
\DeclareMathOperator{\CAT}{CAT}
\makeatletter

\newcommand{\Rom}[1]{\expandafter\@slowromancap\romannumeral #1@}

\numberwithin{equation}{section}
\numberwithin{figure}{section}

\begin{document}

\title[]{On $\CAT(\kappa)$ surfaces}

\author{Saajid Chowdhury}
\address{Department of Physics and Astronomy, Stony Brook University, Stony Brook, NY 11794, USA.}
\email{saajid.chowdhury@stonybrook.edu}

\author{Hechen Hu}
\address{Department of Mathematics, Columbia University, 2990 Broadway, New York, NY 10027, USA.}
\email{hh3032@columbia.edu}

\author{Matthew Romney}
\address{Department of Mathematics, University of Hawaii at Manoa, 2565 McCarthy Mall, Honolulu, HI 96822, USA.}
\email{mromney@hawaii.edu}

\author{Adam Tsou}
\address{Department of Applied Mathematics and Statistics, Johns Hopkins University, 3400 North Charles Street Baltimore, MD 21218, USA.}
\email{atsou2@jh.edu}

\date{\today}
\thanks{}

\begin{abstract}
    We study the properties of $\CAT(\kappa)$ surfaces: length metric spaces homeomorphic to a surface having curvature bounded above in the sense of satisfying the $\CAT(\kappa)$ condition locally. The main facts about $\CAT(\kappa)$ surfaces seem to be largely a part of mathematical folklore, and this paper is intended to rectify the situation. We provide a complete proof that $\CAT(\kappa)$ surfaces have bounded (integral) curvature. This fact allows one to apply the established theory of surfaces of bounded curvature to derive further properties of $\CAT(\kappa)$ surfaces. We also show that $\CAT(\kappa)$ surfaces can be approximated by smooth Riemannian surfaces of Gaussian curvature at most $\kappa$. We do this by giving explicit formulas for smoothing the vertices of model polyhedral surfaces.
\end{abstract}

\maketitle

\section{Introduction}

The field of Alexandrov geometry concerns two main classes of metric spaces: spaces of curvature bounded below and spaces of curvature bounded above. The latter is known as the class of \textit{$\CAT(\kappa)$ spaces}, where $\kappa \in \mathbb{R}$ is some real parameter; these can be defined as spaces in which every sufficiently small geodesic triangle is thinner than the triangle of matching edge lengths in the model space of constant curvature $\kappa$. Alexandrov geometry provides a synthetic or coordinate-free approach to the intrinsic geometry of manifolds or more complicated spaces, generalizing ideas from classical Riemannian geometry to potentially non-smooth spaces. $\CAT(\kappa)$ spaces were first studied systematically by Alexandrov \cite{Alex:57} in the 1950s, building on his investigations on the intrinsic geometry of convex surfaces; see also earlier work of Busemann \cite{Bus:48}.

In the two-dimensional setting, there is a third class of spaces that were extensively studied by Alexandrov and the Leningrad school of geometry associated with him, namely \textit{surfaces of bounded (integral) curvature}. Roughly speaking, this is the largest class of surfaces for which curvature can be meaningfully defined as a signed Radon measure. Its precise definition involves the notion of \textit{(angular) excess} of a geodesic triangle $T$, denoted by $\delta(T)$. A surface of bounded curvature is one for which, in every compactly contained neighborhood, there is a uniform upper bound on the sum of excesses of an arbitrary finite collection of non-overlapping simple triangles. See \Cref{sec:background} and \Cref{sec:cat} for complete definitions. There is a rich and well-developed theory of surfaces of bounded curvature; see the monographs by Alexandrov--Zalgaller \cite{AZ:67} and Reshetnyak \cite{Res:93} covering methods such as polyhedral approximation and conformal parametrizations, as well as the recent surveys \cite{FS:20,Tro:22}. 

We say that a \textit{$\CAT(\kappa)$ surface} is a metric space that is a topological surface and satisfies the $\CAT(\kappa)$ condition locally. The class of $\CAT(\kappa)$ surfaces includes Riemannian surfaces with Gaussian curvature at most $\kappa$ everywhere, and polyhedral surfaces for which the total angle at each vertex is at least $2\pi$. The purpose of this paper is to give complete and accessible proofs for two fundamental facts concerning $\CAT(\kappa)$ surfaces that, to our knowledge, are not satisfactorily addressed in the literature. 

\subsection{The $\CAT(\kappa)$ condition and bounded curvature}

Our first objective is the following. 

\begin{thm} \label{thm:cat_implies_bic}
Let $X$ be a $\CAT(\kappa)$ surface. Then $X$ has bounded curvature.
\end{thm}


This fact is perhaps not as well known as it could be. For example, some of the theorems in the recent papers \cite{BM:21} and \cite{CL:16} on $\CAT(\kappa)$ surfaces are essentially special cases of more general results for bounded curvature surfaces, and the connection between these two classes of surfaces is not made. 
\Cref{thm:cat_implies_bic} is trivial if $\kappa \leq 0$, since in this case every triangle has non-positive excess. In the case where $\kappa>0$ its validity is \textit{a priori} not clear. We note that the analogue of \Cref{thm:cat_implies_bic} for surfaces with curvature bounded below is also true. A proof has been given by Richard in \cite{Ric:18}. 

\Cref{thm:cat_implies_bic} was already observed by Alexandrov himself in \cite[Section 1.8]{Alex:57} in his seminal work on $\CAT(\kappa)$ spaces, justified by an explanation to the effect that one can develop a parallel theory to that of bounded curvature surfaces for surfaces with curvature bounded above by $\kappa$. Alexandrov attributes the result to Zalgaller; however, an elaboration of this approach does not exist in the literature to our knowledge, and in any case seems to be more work than necessary. More recently, a separate proof of \Cref{thm:cat_implies_bic} was sketched by Machigashira--Ohtsuka in \cite{MO:01}. In fact, Lemma 5.3 in \cite{MO:01} states, in more generality, the following basic inequality relating excess and curvature, from which \Cref{thm:cat_implies_bic} follows easily.

\begin{thm} \label{thm:excess_curvature}
Let $T$ be a triangle in a $\CAT(\kappa)$ surface that is the boundary of a closed disk. Then $\delta(T) \leq \kappa |T|$.
\end{thm}
Here, $|T|$ denotes the area of the interior of $T$, which we define as the Hausdorff $2$-measure of this set. According to \cite{MO:01}, \Cref{thm:excess_curvature} can be established in a similar manner to an analogous inequality for surfaces with curvature bounded below, which was done in earlier work of Machigashira \cite[Theorem 2.0]{Mach:98}. However, there are various qualitative differences between $\CAT(\kappa)$ spaces and spaces of curvature bounded below, and thus we feel that it is not straightforward to adapt the proof in \cite{Mach:98} to the $\CAT(\kappa)$ case.

In light of the state of the literature, Fillastre asked in \cite{Fil:MO} for complete details of a proof. In this paper, we answer the request by giving a complete proof of \Cref{thm:cat_implies_bic}. Our approach is based on a recent theorem of the third-named author with Creutz \cite{CR:22} on decomposing an arbitrary surface with a length metric into non-overlapping simple convex triangles, which we refer to as a \textit{triangulation} of the the surface. See \Cref{thm:triangulation} below. The main tool we use is a special type of triangulation that can be obtained by inductively subdividing a base triangulation, which we call a \textit{vertex-edge triangulation}. 

We remark that $\CAT(\kappa)$ surfaces can be equivalently defined as limits of smooth or polyhedral surfaces of locally uniformly bounded curvature. As an alternative proof of \Cref{thm:cat_implies_bic}, it is possible to use the triangulation theorem in \cite{CR:22} to verify this equivalent condition. However, the proof we give in this paper seems to us the most intuitive and comprehensive. 

Later in the paper, we show how \Cref{thm:excess_curvature} itself follows readily as a consequence of \Cref{thm:cat_implies_bic} using the theory of surfaces of bounded curvature. One technical point relates to the definition of \textit{area}. As stated above, we take the Hausdorff $2$-measure as our definition. The original work of Alexandrov (cf. \cite{Alex:57}) uses a different definition based on approximation by Euclidean polyhedral surfaces, which we call the \textit{Alexandrov area}. We verify that the two notions of area are equivalent for $\CAT(\kappa)$ surfaces, or more generally for surfaces of bounded curvature without cusp points:
\begin{prop} \label{prop:hausdorff_vs_alexandrov}
For any surface of bounded curvature without cusp points, the Alexandrov area coincides with the Hausdorff $2$-measure.  
\end{prop}


\subsection{Smooth approximation}
The second objective of this paper concerns approximation by smooth Riemannian surfaces. We prove the following result.

\begin{thm} \label{thm:smooth_approx}
Let $(X,d)$ be a $\CAT(\kappa)$ surface for some $\kappa \in \mathbb{R}$. Then there exists a sequence of smooth Riemannian metrics $d_n$ on $X$ with Gaussian curvature at most $\kappa$ such that the sequence $(X,d_n)$ converges uniformly to $(X,d)$. Moreover, the metrics $d_n$ have the property that $\limsup_{n \to \infty} |A_n| = |A|$ for all compact sets $A \subset X$. Here, $A_n$ denotes the same set $A$ equipped with the metric $d_n$.
\end{thm}

While one naturally expects \Cref{thm:smooth_approx} to be true, we have not found any complete statement in the literature. From the theory of surfaces of bounded curvature comes the fact that any $\CAT(\kappa)$ surface can be approximated by smooth surfaces with locally uniformly bounded integral curvature, though without any pointwise control on the curvature. See Chapter III of \cite{AZ:67} and Section 6.2 of \cite{Res:93}. In \cite[Theorem 0.11]{BuBu:98}, Burago--Buyalo show that any $\CAT(\kappa)$ surface can be approximated by piecewise smooth $\CAT(\kappa)$ surfaces (more generally, their result covers any 2-dimensional locally $\CAT(\kappa)$ polyhedron). Recently, the case of \Cref{thm:smooth_approx} where $\kappa<0$ has been covered by a lemma of Labeni \cite{Lab:21}. Labeni's argument is based on the approximating cones with vertices of negative curvature in anti-de Sitter space by smooth convex surfaces.

To prove \Cref{thm:smooth_approx}, the main task is to show that any vertex of a $\CAT(\kappa)$ polyhedral surface can be smoothed while retaining the $\CAT(\kappa)$ condition. We found that this can be accomplished via an explicit and relatively simple formula that interpolates between a metric of constant positive curvature away from the vertex with a metric of constant negative curvature around the vertex.

The analogue of \Cref{thm:smooth_approx} for surfaces of curvature bounded below has been addressed in full generality by Itoh--Rouyer--V\^{\i}lcu \cite{IRV:15}; see also \cite{Ram:15,Ric:18} and the discussion in \cite{Kap:MO,makt:MO} concerning the curvature bounded below case. Our approach also works for smoothing surfaces of curvature bounded below by $\kappa$, with the advantage of being more explicit than the approach in \cite{IRV:15}; see \Cref{rem:cbb_case}. Thus we provide a concrete and uniform approach to the general problem of smoothing vertices of polyhedra.

\subsection{The bigger picture} \label{sec:bigger}

We now elaborate more on the consequences of \Cref{thm:cat_implies_bic} and place our work within a larger context. While the $\CAT(\kappa)$ condition provides an effective axiomatic condition for studying metric surfaces, there is a variety of increasingly general conditions that have also been studied in the literature. Indeed, the $\CAT(\kappa)$ condition is fairly restrictive; roughly speaking, a $\CAT(\kappa)$ surface is one that is the limit of spherical polyhedral surfaces of uniformly bounded area with vertices of negative curvature. We state the following omnibus theorem, which is a compilation of easy or known facts, along with \Cref{thm:cat_implies_bic}. For simplicity, we assume that the metric surface $X$ is closed. 

\begin{thm} \label{thm:omni}
Let $X$ be a closed length surface and $\kappa \in \mathbb{R}$. The following conditions are such that each condition implies the next.
\begin{enumerate}[label=(\Alph*)]
    \item \label{item:A} $X$ is a smooth Riemannian $2$-manifold with Gaussian curvature at most $\kappa$ everywhere.
    \item \label{item:B} $X$ is a $\CAT(\kappa)$ surface.
    \item \label{item:C} $X$ is a surface of bounded curvature without cusp points.
    \item \label{item:D} $X$ is bi-Lipschitz equivalent to a constant curvature surface of the same topology. In particular, each point of $X$ has a neighborhood bi-Lipschitz equivalent to a Euclidean disk.
    \item \label{item:E} $X$ is Ahlfors $2$-regular and linearly locally contractible.
    \item \label{item:F} $X$ satisfies a quadratic isoperimetric inequality.
    \item \label{item:G} $X$ has locally finite Hausdorff $2$-measure.
\end{enumerate}
\end{thm}
Let us define the terms in this theorem that are potentially less familiar or standardized. A metric space $X$ is \textit{Ahlfors $2$-regular} if there is a constant $C > 0$ such that
\[\frac{1}{C} \cdot r^2 \leq |B(x,r)|\leq C \cdot r^2 \]
for all $x \in X$ and $r \in (0, \diam(X))$. Here, $|A|$ denotes the Hausdorff $2$-measure of a set A and $B(x,r)$ is the open ball at $x$ of radius $r>0$. The space $X$ is \textit{linearly locally contractible} if  there exists $C>0$ such that every ball $B(x,r)$ of radius $r \in (0,\diam(X)/C)$ is contractible in $B(x,Cr)$. We say that the space $X$ \textit{satisfies a quadratic isoperimetric inequality} if every point has a neighborhood $V$ for which there is a constant $C>0$ such that every closed Jordan curve $\Gamma$ in $V$ bounds a topological disk $U \subset V$ satisfying $|U| \leq C \ell(\Gamma)^2$. Here $\ell(\Gamma)$ is the length of $\Gamma$.

The implication \ref{item:A} $\Longrightarrow$ \ref{item:B} is classical; see the appendix to Chapter II.1 of \cite{BH:99}. The implication \ref{item:B} $\Longrightarrow$ \ref{item:C} is a restatement of \Cref{thm:cat_implies_bic} together with the simple observation that a $\CAT(\kappa)$ surface does not have any so-called \textit{cusp points}, i.e., points having total curvature $2\pi$. Next, the implication \ref{item:C} $\Longrightarrow$ \ref{item:D} follows from a theorem of Reshetnyak that surfaces of bounded curvature without cusp points can be approximated in the bi-Lipschitz sense by polyhedral surfaces;
 see Theorems 9.10 and 9.11 in \cite{Res:93} and Lemma 6 in \cite{Bur:04}. See also Bonk--Lang \cite{BL:03} for a strong quantitative version of this property. 
The implications \ref{item:D} $\Longrightarrow$ \ref{item:E} and \ref{item:F} $\Longrightarrow$ \ref{item:G} are elementary to verify from the definitions, while the implication \ref{item:E} $\Longrightarrow$ \ref{item:F} is due to Lytchak--Wenger \cite[Corollary 5.5]{LW:20}.

The final condition \ref{item:G}, namely, having locally finite Hausdorff $2$-measure, represents the most general non-fractal metric surface. We briefly remark that fractal surfaces are a separate topic of current interest; these are generally more difficult to study and require separate methods. See for example the monograph of Bonk--Meyer \cite{BM:17}. Condition \ref{item:G} is very general; nevertheless, the basic idea of the recent papers \cite{CR:22,NR:21,NR:22} is that a geometric study of surfaces analogous to that of surface of bounded curvature can be carried out using only the assumption of locally finite Hausdorff $2$-measure. In fact, it turns out that the requirement of a length metric is not so essential and can be omitted. 

Regarding condition \ref{item:F}, we mention a beautiful theorem of Lytchak--Wenger \cite{LW:18b} that upper curvature bounds can be characterized in terms of isoperimetric inequalities. Namely, a proper length metric space is $\CAT(0)$ if and only if it satisfies the quadratic isoperimetric inequality (defined in a more general way) with the Euclidean constant $1/(4\pi)$. More generally, a proper length metric space is $\CAT(\kappa)$ if and only if its isoperimetric profile is no worse than that of the sphere of constant curvature $\kappa$. 

One aspect of the theory concerns the topic of uniformization. The classical uniformization theorem states that any smooth Riemannian $2$-manifold is conformally equivalent to a surface of constant curvature. A major theorem of Reshetnyak concerns conformal parametrizations of bounded curvature surfaces: for any surface of bounded curvature, its metric can be represented locally by a length element of the form $ds = e^{\lambda}ds_h$, where $\lambda$ is the difference of two subharmonic functions and $ds_h$ is the length element of a background smooth Riemannian metric \cite{Res:60}. The global version of this theorem is due to Huber \cite{Hub:60}. There are similar uniformization theorems for the more general classes of surfaces in \Cref{thm:omni}, including the Bonk--Kleiner theorem for Ahlfors $2$-regular, linearly locally contractible spheres \cite{BK:02}. A very general uniformization theorem applicable to any surface of locally finite Hausdorff $2$-measure, yielding a so-called \textit{weakly quasiconformal} parametrization from a constant curvature surface, has been proved in \cite{NR:21,NR:22}.

We refer the reader to the following references for various non-trivial examples of surfaces illustrating the difference between the classes of surfaces in \Cref{thm:omni}: Laakso \cite{Laa:02} for a surface satisfying \ref{item:E} but not \ref{item:D}; Theorem 1.1 in \cite{CR:20} for a surface satisfying \ref{item:F} but not \ref{item:E}; Theorem 1.2 in \cite{Rom:19} for a surface satisfying \ref{item:G} but not \ref{item:F}. In addition, we observe that $\mathbb{R}^2$ equipped with any normed metric not arising from an inner product satisfies \ref{item:D} but not \ref{item:C}. Moreover, any polyhedral surface with cone points of positive curvature satisfies \ref{item:C} but not \ref{item:B}.

Finally, we remark that there have been several more advanced investigations of finite-dimensional $\CAT(\kappa)$ spaces in recent years; see especially \cite{LN:19,LN:22}. Two-dimensional metric spaces (not necessarily surfaces) satisfying the $\CAT(\kappa)$ condition have been studied in \cite{NSY:22} and \cite{NSY:23}. 

\subsection{Outline of the paper}

The paper is organized as follows. We collect preliminaries related to metric geometry in \Cref{sec:background} and related to $\CAT(\kappa)$ spaces in \Cref{sec:cat}. The material on $\CAT(\kappa)$ spaces is probably well-known to experts but includes some propositions that we could not find in the literature. \Cref{sec:vertex_edge} describes our main tool of vertex-edge triangulations in preparation for the proof of \Cref{thm:cat_implies_bic}. The proof of \Cref{thm:cat_implies_bic} is given in \Cref{sec:bic_proofs}. 
The other results are proved in the following sections. First, we verify \Cref{prop:hausdorff_vs_alexandrov} on definitions of area and \Cref{thm:excess_curvature} relating the excess of a triangle to its area in \Cref{sec:excess_curvature}. Finally, we prove \Cref{thm:smooth_approx} on smooth approximation in \Cref{sec:smooth_approximation}.

\section*{Acknowledgments}

The research for this project was carried out as part of Stony Brook University's Research Experience for Undergraduates (REU) program during the summer of 2022. The authors thank the Stony Brook University Mathematics Department for their support. The authors also thank Fran\c cois Fillastre and Alexander Lytchak for their encouragement in completing this project and for feedback on an early draft of this paper. We also especially thank Fran\c cois Fillastre for bringing to our attention several references in the literature.

\section{Background} \label{sec:background}

\subsection{Basics of metric geometry}
We recall the basic definitions related to metric spaces. See standard references such as \cite{BH:99} and \cite{BBI:01} for more detail. Let $X$ be a set. A \textit{metric} on $X$ is a function $d \colon X \times X \to \mathbb{R}$ such that for all $x,y,z \in X$ 
\begin{enumerate}
    \item $d(x,y) = 0$ if and only if $x=y$ (positive definite),
    \item $d(x,y) = d(y,x)$ (reflexive),
    \item $d(x,y) \leq d(x,z) + d(z,y)$ (triangle inequality).
\end{enumerate}
The pair $(X,d)$ is called a \textit{metric space}. Most often, we refer to $X$ itself as a metric space with the convention that $d$ denotes the corresponding metric. We refer to $d(x,y)$ as the \textit{distance} between $x$ and $y$. Given $x \in X$ and $r>0$, we let $B(x,r)$ denote the open metric ball centered at $x$ of radius $r$. Note that this notation does not specify the metric space, though no confusion should result.

A \textit{curve} in $X$ is a continuous function from an interval to $X$. The \textit{length} of a curve $\Gamma \colon [a,b] \to X$ is
\[\ell(\Gamma) = \mathrm{sup}\sum_{k=1}^m d(\Gamma(t_{k-1}),\Gamma(t_k)),\]
where the supremum is taken over all finite subdivisions $a = t_0 < t_1 < \cdots < t_m = b$. The curve $\Gamma$ is a \textit{geodesic} if $\ell(\Gamma) = d(\Gamma(a),\Gamma(b))$. The notation $[xy]$ is used to denote a choice of geodesic from $x$ to $y$, provided that such a geodesic exists. Note that in general such a geodesic need not be unique. In a slight abuse of notation, we also identify a geodesic $[xy]$ with its image.

The metric space $X$ is a \textit{length space} if 
\[d(x,y) = \inf_{\Gamma} \ell(\Gamma)\]
for all $x,y \in X$, where the infimum is taken over all curves $ \Gamma $ connecting $x$ and $y$. The space $X$ is a \textit{geodesic space} if every pair of points $x,y$ can be joined by a geodesic $[xy]$. Observe in this case that $d(x,y) = \ell([xy])$. 

Let $(X,d)$, $(Y,d')$ be metric spaces and $L>0$. A map $f \colon X \to Y$ is \textit{$L$-Lipschitz} if
\[
d'(f(p),f(q)) \leq L\cdot d(p,q)
\]
for all $p, q \in X$. The map $f$ is \textit{$L$-bi-Lipschitz} if 
\[
\frac{1}{L}d(p,q)\leq d'(f(p),f(q)) \leq Ld(p,q)
\]
for all $p,q \in X$. 

The diameter of a set $A \subset X$ is 
\[\diam A = \sup_{x,y \in A} d(x,y). \] 
The \textit{$2$-dimensional Hausdorff $\delta$-content} of $A \subset X$ is 
\[ \mathcal{H}_{\delta}^2(A) = \mathrm{inf} \left \{ \sum_{j=1}^\infty c_2\cdot \mathrm{diam}(A_j)^2\right \},\]
where the infimum is taken over all countable collections of subsets $A_1, A_2, \ldots$ of diameter at most $\delta$ such that $A \subset \bigcup_{j=1}^\infty A_j$. The \textit{Hausdorff $2$-measure} of $A$ is
\[\mathcal{H}^2(A) = \lim_{\delta \to 0} \mathcal{H}_{\delta}^2(A).\]
Here, $c_2 = \pi/4$ is a normalization constant so that the Hausdorff $2$-measure on the plane coincides with $2$-dimensional Lebesgue measure. Accordingly, we denote the Hausdorff $2$-measure of a set $A$ by $|A|$ rather than $\mathcal{H}^2(A)$ throughout this paper. The term ``area'' in this paper refers specifically to the Hausdorff $2$-measure. 

\subsection{Polygons and triangulations}

We start with a general definition of polygon.

\begin{defi} \label{defi:polygon}
 Let $X$ be a length metric space and $n \geq 2$ an integer. A \textit{(geodesic) $n$-gon} is a figure consisting of $n$ points $v_1,\cdots,v_n$, called \textit{vertices}, together with a collection of $n$ geodesics joining them cyclically, called \textit{edges}. An $n$-gon is also called a \textit{polygon}. A polygon is \textit{simple} if it forms a simple closed curve. A $3$-gon with vertices $x,y,z$ is called a \textit{triangle} and denoted by $\triangle xyz$. A $4$-gon is called a \textit{quadrilateral}. The \textit{perimeter} of a polygon is the sum of the lengths of its edges. 
\end{defi}

The previous definition is valid for any metric space, although our focus in this paper is two-dimensional metric spaces. A \textit{metric surface} is a metric space homeomorphic to a topological surface, or $2$-manifold. Consistent with the main references such as \cite{AZ:67}, we assume that metric surfaces do not have boundary, although in principle the results of this paper should remain true for surfaces with boundary. We use the terms \textit{length surface} and \textit{geodesic surface} with the obvious meanings.

We say that a simple polygon $P$ is a \textit{Jordan polygon} if it is contained in a neighborhood $U$ homeomorphic to an open disk. 
By the Jordan curve theorem, any simple polygon $P$ is the boundary of a closed Jordan domain contained in the ambient neighborhood $U$. The interior of this domain is called the \textit{interior} of the polygon, denoted by $P^\circ$. Note that the previous fact depends on our assumption of an ambient disk neighborhood; otherwise, a simple polygon may not bound a unique disk, or any disk at all. For example, a great circle on the $2$-sphere can be viewed as a triangle, and either hemisphere may be designated as the interior. 

\begin{conv}
    When we refer to a ``Jordan polygon'' $P$, as opposed to simply a ``polygon'', we include the interior points of $P$. In turn, we use $\partial P$ to denote the set of edge points of $P$.
\end{conv}






Two Jordan polygons are \textit{non-overlapping} if they have disjoint interiors. A Jordan polygon $P \subset X$ is \textit{convex} if any two points in $P$ can be joined by a geodesic in $P$. This definition of convexity is sufficient for this paper, although we note that there are stronger notions of convexity that are essential for the study of general length surfaces. See Section 3 of \cite{CR:22} for more discussion.

We state a simple property of convexity.
\begin{prop} \label{prop:convex_intersection}
Let $P,S$ be two convex Jordan polygons in a metric surface with the property that any two points in $P \cup S$ can be joined by a unique geodesic. If $P^\circ \cap S^\circ$ is non-empty, then $P \cap S$ is a convex Jordan polygon.
\end{prop}
See Lemma 3.2 in \cite{CR:22} for a closely related proposition. Note that the conclusion may fail without the unique geodesic property.

Our proofs make use of covers of a surface by non-overlapping triangles, which we call \textit{triangulations}. Our definition is somewhat different than the use of the term in general topology, since we do not require adjacent triangles to intersect along entire edges. 

\begin{defi}
    A \textit{triangulation} of a subset $K \subset X$ is a locally finite collection of mutually non-overlapping convex Jordan triangles $\mathcal{T}=\{T_i\}$ such that $\bigcup_i T_i=K$. 
    If the subset $K$ itself is a polygon, then the vertices of $K$ are called the \textit{original vertices} of the triangulation $\mathcal{T}$. The \textit{mesh} of $\mathcal{T}$ is $\sup_{T \in \mathcal{T}} \diam T$.
\end{defi}

The union of the set of edges of a triangulation $\mathcal{T}$ form an embedded topological graph on the surface, called the \textit{edge graph} of $\mathcal{T}$ and denoted by $\mathcal{E}$. The edge graph can be considered as a length space by giving it the length metric induced by the metric on $X$.

A general theorem on the existence of triangulations for arbitrary length surfaces was proved in \cite{CR:22}. We state a version that covers our purposes.

\begin{thm} \label{thm:triangulation}
Let $X$ be a length surface and $\mathcal{U}$ a cover of $X$ by open sets. For all $\varepsilon>0$ there exists a triangulation $\mathcal{T}$ of $X$ with mesh at most $\varepsilon$ such that each element $T\in \mathcal{T}$ is contained in a set $U \in \mathcal{U}$ satisfying $d(T,\partial U) > \diam T$. Moreover, if $\mathcal{E}$ is the edge graph of $\mathcal{T}$ with the induced length metric, then the inclusion of $\mathcal{E}$ into $X$ is an $\varepsilon$-isometry. 

The same result holds for surfaces with boundary provided the boundary is piecewise geodesic.
\end{thm}
A map $f \colon X \to Y$ between metric spaces $(X,d)$ and $(Y,d')$ is an \textit{$\varepsilon$-isometry} if 
\[d(x,y) - \varepsilon \leq d'(f(x),f(y)) \leq d(x,y) + \varepsilon\] 
for all $x,y \in X$ and the image of $f$ is $\varepsilon$-dense in $Y$ (that is, every point in $y$ is at distance at most $\varepsilon$ from a point in $f(X)$).  

The statements concerning the open cover $\mathcal{U}$ is not part of the triangulation theorem as stated in Theorem 1.2 of \cite{CR:22}, although it follows from Remark 5.5 in \cite{CR:22}. Also, the triangulation theorem in \cite{CR:22} gives the stronger property of \textit{boundary convexity} rather than just convexity, though this is not needed for our purposes. The final conclusion regarding the edge graph can be found in Proposition 5.2 in \cite{NR:21}. 

\section{Preliminaries on $\CAT(\kappa)$ spaces} \label{sec:cat}

\subsection{Basic definitions}
In this section, we give a brief overview of $\CAT(\kappa)$ surfaces and collect various facts that are needed later. We refer the reader to Bridson--Haefliger \cite{BH:99} for further details.

For each $\kappa \in \mathbb{R}$, we let $M_\kappa^2$ denote the $2$-dimensional model space of constant curvature $\kappa$. More precisely, these are defined by letting $M_0^2$ be the Euclidean plane, $M_1^2$ be the $2$-sphere, and $M_{-1}^2$ be the hyperbolic plane. For $\kappa>0$, $M_{\pm \kappa}^2$ is the rescaling of $M_{\pm 1}^2$ by the factor $\sqrt{\kappa}$.  We let $\bar{d}_\kappa$ denote the metric on $M_\kappa^2$ and $D_\kappa$ denote the diameter of $M_\kappa^2$; that is, $D_\kappa = \pi/\sqrt{\kappa}$ if $\kappa>0$ and $D_\kappa = \infty$ otherwise. It is often convenient to think of each model space as a conformal deformation of the Euclidean metric, that is, a smooth Riemannian metric obtained by rescaling pointwise the Euclidean metric on a domain in $\mathbb{R}^2$. Let $ds_\kappa$ denote the corresponding length element for $M_\kappa^2$; then $ds_\kappa$ is given by the formula
\[ds_\kappa^2 =  \frac{4(dx^2+dy^2)}{(1+\kappa(x^2+y^2))^2}\]
defined for all $(x,y) \in \mathbb{R}^2$, with the additional requirement that $x^2 + y^2 < 1/|\kappa|$ if $\kappa<0$. If $\kappa=0$, this formula gives a rescaling of the usual Euclidean metric, but the resulting metric space is of course isometric to the Euclidean plane. If $\kappa>0$, then we must add in the point $\infty$ to obtain the complete $2$-sphere. 

    Next, let $X$ be a metric space and fix $\kappa \in \mathbb{R}$. Given a triangle $T=\triangle pqr\subset X$ with perimeter at most $2D_\kappa$, a \textit{comparison triangle in $M^2_\kappa$} for $T$ is a triangle in $M^2_\kappa$ with edges of matching lengths. We denote the comparison triangle of $T$ by $\bar{T}_\kappa$. More precisely, let $\bar{p}, \bar{q}, \bar{r} \in M_\kappa^2$ be points satisfying $d(p,q) = \bar{d}_\kappa(\bar{p},\bar{q})$, $d(p,r) = \bar{d}_\kappa(\bar{p},\bar{r})$ and $d(q,r) = \bar{d}_\kappa(\bar{q},\bar{r})$. Then we define $\bar{T}_\kappa = \triangle \bar{p}\bar{q}\bar{r}$. This triangle is unique up to an isometry of $M^2_\kappa$.  
    
    Given a point $x \in [pq]$, we let $\bar{x}$ denote the point on $[\bar{p}\bar{q}]$ such that $d(p,x)=\bar{d}_{\kappa}(\bar{p},\bar{x})$ and $d(x,q)=\bar{d}_{\kappa}(\bar{x},\bar{q})$. A triangle $T=\triangle pqr\subset X$ with perimeter at most $2D_\kappa$ is said to \textit{satisfy the $\CAT(\kappa$) inequality} if 
    \[d(x,y)\leq \bar{d}_{\kappa}(\bar{x},\bar{y})\]
    for all $x,y$ in $\partial T$. A length metric space $X$ is a \textit{$\CAT(\kappa)$ space} if any two points in $X$ at distance less than $D_\kappa$ are joined by a geodesic and every triangle $T \subset X$ with perimeter less than $2D_\kappa$ satisfies the CAT($\kappa$) inequality. We say that a length surface $X$ is a \textit{locally $\CAT(\kappa)$} surface if every point $x \in X$ has a neighborhood $U$ that is a $\CAT(\kappa)$ space as previously defined. 
    For conciseness, we typically omit the word \textit{locally} and refer to such a space as a \textit{$\CAT(\kappa)$ surface}. 

    \subsection{Angles}
    Let $X$ be an arbitrary metric space. Given a triangle $T = \triangle pqr\subset X$, the \textit{model angle at $p$}, denoted $\bar{\angle}_{p}^\kappa(q,r)$, is the angle between the geodesics $[\bar{p}\bar{q}]$ and $[\bar{p}\bar{r}]$ in the comparison triangle $\triangle\bar{p}\bar{q}\bar{r}$ in $M^2_\kappa$. This is defined provided that the perimeter of $T$ is at most $2D_\kappa$. In a $\CAT(\kappa)$ space, the model angle satisfies the triangle inequality: $\bar{\angle}_{p}^\kappa(q,r) \leq \bar{\angle}_{p}^\kappa(q,s)+\bar{\angle}_{p}^\kappa(s,r)$ for all $p,q,r,s$, provided that all these angles are defined. 
    
    For any metric space $X$, the \textit{(Alexandrov) upper angle} between geodesics $[pq]$ and $[pr]$ is
    \begin{equation} \label{equ:upper_angle}
        \angle_p(q,r)=\limsup_{q',r' \to p} \bar{\angle}_{p}^0 (q',r')    
    \end{equation}
    where $q'\in [pq] \setminus \{p\}$ and $r'\in [pr]\setminus \{p\}$. In general, the upper angle satisfies the triangle inequality: $\angle_{p}(q,r) \leq \angle_{p}(q,s)+\angle_{p}(s,r)$ for all points $p,q,r,s$ such that the geodesics $[pq], [pr], [ps]$ exist. If the ``$\limsup$'' in \eqref{equ:upper_angle} is an actual limit, then $\angle_p(q,r)$ is called the \textit{angle} between $[pq]$ and $[pr]$. If $X$ is a $\CAT(\kappa)$ space, the angle between adjacent sides of any geodesic triangle of perimeter less than $2D_\kappa$ always exists and is less than or equal to the corresponding model angle. 
    
    Let $T$ be a triangle with angles $\alpha, \beta, \gamma$. 
    The \textit{angle excess} or \textit{excess} of $T$ is $\delta(T) = \alpha + \beta + \gamma -\pi$. If $\mathcal{T}$ is a finite family of geodesic triangles, the \textit{excess} of $\mathcal{T}$ is $\delta(\mathcal{T}) = \sum_i \delta(T_i)$.
    In connection with excess, we recall the classical Girard's theorem, which states that the area of a triangle in the $2$-sphere $M_1^2$ is precisely equal to its excess. More generally, if $T$ is a triangle in the model space $M_\kappa^2$ for any $\kappa \in \mathbb{R}$, then $\kappa |T| = \delta(T)$. 
    
    \subsection{Facts about $\CAT(\kappa)$ spaces} \label{sec:cat_facts}
    
    We record some properties of $\CAT(\kappa)$ spaces. The first concerns uniqueness of geodesics. If two points $x,y \in X$ satisfying $d(x,y) < D_\kappa$ are joined by a geodesic, then that geodesic is unique. The existence of a geodesic between sufficiently close points is a consequence of the Hopf--Rinow theorem; see Proposition I.3.7 in \cite{BH:99}. Suppose that $x,y$ belong to a neighborhood $U \subset X$ homeomorphic to a closed disk. If $d(x,y) \leq d(x,\partial U)$, then there is a geodesic in $X$ from $x$ to $y$. Moreover, the ball $B(x,r)$ is convex and contractible for all $r< \min\{d(x,\partial U)/2, D_\kappa/2\}$. See Proposition II.1.4 in \cite{BH:99}.

\begin{lemm}
\label{lemm:extend}
Let $X$ be a $\CAT(\kappa)$ space that is also a topological manifold. Then every geodesic can be extended indefinitely in both directions as a local geodesic. It is length-minimizing for all pairs of points at distance at most $D_\kappa$.  
\end{lemm}

See Proposition II.$5.12$ in \cite{BH:99} for a proof. We can use geodesic extendability to show to following.

\begin{cor} \label{cor:extension}
Let $X$ be a $\CAT(\kappa)$ space that is also a topological surface. Let $T=\triangle pqr$ be a Jordan triangle in $X$ of diameter less than $D_\kappa$ and let $x$ be an interior point. Then the geodesic $[px]$ can be extended to intersect the opposite side $[qr]$.
\end{cor}
\begin{proof}
By \Cref{lemm:extend}, $[px]$ can be extended indefinitely past the point $x$. Let $\Gamma$ denote a maximal extension past $x$, which we assume is parametrized by arc length. We claim that $\Gamma$ must leave the set $T$. Assume this is not the case. If $\Gamma$ has finite length, then its domain must be an interval of the form $[a,b)$. Since $\Gamma$ is $1$-Lipschitz and $T$ is compact, $\lim_{t \to b} \Gamma(t)$ must exist and lie in $T$. If we set $\Gamma(b) = \lim_{t \to b} \Gamma(t)$, then by continuity $\Gamma$ retains the property of being a geodesic. This contradicts the maximality of $\Gamma$. Hence $\Gamma$ must have infinite length. However, since $T$ has diameter at most $D_\kappa$, it must be the case that $\Gamma$ is length minimizing between all pairs of points it contains. Again since $T$ has diameter at most $D_\kappa$, this contradicts the statement that $\Gamma$ has infinite length. 
We conclude that $\Gamma$ intersects $\partial T$ at some point other than $p$. This point must lie on the side $[yz]$ since otherwise the uniqueness of geodesic is violated.
\end{proof}

Likewise, small Jordan triangles in $\CAT(\kappa)$ surfaces are convex:
\begin{lemm} \label{lemm:convexity}
    Let $X$ be a $\CAT(\kappa)$ space that is also a topological surface, and let $T$ be a Jordan triangle in $X$. Assume further that $T$ is contained in a neighborhood $U$ homeomorphic to a closed disk with $d(T,\partial U) \geq \ell(\partial T)$ and $\text{diam}(U) < D_\kappa$. Then $T$ is convex.
\end{lemm}
\begin{proof}
Consider two points $x,y \in \partial T$. Observe that $d(x,y) < \ell(\partial T)$ and that $B(x, \ell(\partial T)) \subset U$. This implies that there is a geodesic $\Gamma$ from $x$ to $y$ in $U$, which moreover is unique since $\diam(U) < D_\kappa$. We must show that $\Gamma$ is contained in $T$. If $x,y$ belong to the same side of $T$, then $\Gamma$ is just a subarc of this side and we are done. Assume then that $x$ belongs to the side $A$ and $y$ belongs to the side $B$, with $C$ the other side of $T$. By arguing as in \Cref{cor:extension}, the geodesic $A$ can be extended in each direction until reaching the boundary $\partial U$. Since $\diam(U) < D_\kappa$, this extension is also a geodesic, which we still denote by $A$. This extension divides $U$ into two non-overlapping closed disks $U_A^1$ and $U_A^2$, where $T \subset U_A^1$. Do the same for $B$ to obtain two non-overlapping closed disks $U_B^1$ and $U_B^2$, with $T$ contained in $U_B^1$, and likewise for $C$. By the uniqueness of geodesics, $\Gamma$ cannot cross the geodesic $A$, and hence $\Gamma$ must be contained in the set $U_A^1$. Likewise, $\Gamma$ must be contained in the sets $U_B^1$ and $U_C^1$. Since $T = U_A^1 \cap U_B^1 \cap U_C^1$, we deduce that $\Gamma$ lies in $T$. Thus $T$ is convex.
\end{proof}

\subsection{Surfaces of bounded curvature} \label{sec:surfaces_bounded_curvature}

In this section, we give a brief overview of the theory of surfaces of bounded curvature. See the monographs \cite{AZ:67} and \cite{Res:93} for a detailed treatment, as well as \cite{Tro:22} for a shorter survey. 
The fundamental definition in the theory as developed by Alexandrov--Zalgaller is the following. 

\begin{defi}
A length surface $X$ has \textit{bounded curvature} if for each $x \in X$ there is a neighborhood $U$ containing $x$ and a constant $C>0$ such that $\delta(\mathcal{T}) \leq C$ for all finite collections $\mathcal{T}$ of mutually non-overlapping Jordan triangles contained in $U$.
\end{defi}

The main method of proof used by Alexandrov--Zalgaller is polyhedral approximation. For a flat polyhedral surface $Y$ (i.e., a polyhedral surface where each face is a Euclidean polygon), one can define the total curvature at a vertex to be $2\pi$ minus the total angle at that vertex. One can then define the curvature measure $\omega$ on $Y$ as a signed Radon measure by defining $\omega(A)$ to be the sum of the total curvatures of all vertices in $A$. The measure $\omega$ in turn can be split into a positive part $\omega^+$ and negative part $\omega^-$, where $\omega^+$ and $\omega^-$ are non-signed measures satisfying $\omega = \omega^+ - \omega^-$. Suppose that $(X,d)$ is a length surface and $(d_n)$ is a sequence of polyhedral metrics on $X$ converging uniformly to $d$, with $\omega_n$ the corresponding curvature measures. We say that the sequence of metrics $d_n$ has \textit{locally uniformly bounded curvature} if for any compact set $A \subset X$ both $\omega_n^+(A)$ and $\omega_n^-(A)$ are uniformly bounded in $n$. 

Surfaces of bounded curvature can be defined equivalently in terms of a sequence of polyhedral approximations. Namely, a length surface $(X,d)$ has bounded curvature if and only if there is a sequence of polyhedral metrics $d_n$ on $X$ with locally uniformly bounded curvature that converges uniformly to $d$. The same statement is true with ``polyhedral metric'' replaced by ``Riemannian metric''; in this case, the curvature measure is defined by integrating the Gaussian curvature with respect to area. Surfaces of bounded curvature can also be characterized in terms of the existence of conformal parametrizations due to Reshetnyak, which we discussed briefly in \Cref{sec:bigger}.

The characterization as limits of polyhedral surfaces enables one to define area and the curvature measure on a surface of bounded curvature by taking weak limits. We call the notion of area obtained this way the \textit{Alexandrov area} and denote the area of a set $A$ by $\Area(A)$. More precisely, let $(d_n)$ be a sequence of polyhedral metrics on $X$ with locally uniformly bounded curvature converging uniformly to a metric $d$, with $\varphi_n$ denoting the transition map from $(X,d_n)$ to $(X,d)$ and $\omega_n$ the curvature measure for $d_n$. For all Borel sets $A \subset X$ we define $\Area(A) = \lim_{n\to \infty} |\varphi_n^{-1}(A)|$. Likewise, we define $\omega(A) = \lim_{n\to \infty} \omega_n(\varphi_n^{-1}(A))$ It can be shown that these are independent of the choice of approximating surface and hence well-defined; see Section VIII.3 of \cite{AZ:67}. 

\section{Vertex-Edge Triangulations} \label{sec:vertex_edge}

In this section, we define vertex-edge triangulations and establish their main properties. Throughout this section, $X$ is a $\CAT(\kappa)$ surface. 

All polygons we consider for the remainder of the paper are Jordan polygons, so we drop the word ``Jordan'' and just use the term ``polygon''. We recall our convention that interior points are considered part of a (Jordan) polygon.

First, we formally state the relevant definitions.

\begin{defi}
    Let $P \subset X$ be a convex polygon. A \textit{vertex-edge partition} of $P$ is a cover of $P$ by non-overlapping convex polygons that can be obtained by the following inductive process. First, $\mathcal{T}_0 = \{P\}$ itself is a vertex-edge partition.
    \begin{itemize}[leftmargin=5.mm]
        \item[] \textit{Vertex-edge subdivision.} Let $\mathcal{T}$ be a vertex-edge partition. Given a polygon $T \in \mathcal{T}$, we subdivide it by connecting one of its vertices to a point on an edge not containing that vertex by a geodesic in $T$ that intersects $\partial T$ only at the endpoints. This produces two non-overlapping convex polygons $T_1$ and $T_2$ whose union is $T$. We define a new vertex-edge partition $\mathcal{T}'$ by replacing $T$ with the two polygons $T_1,T_2$.   
    \end{itemize}
    If a vertex-edge partition $\mathcal{T}$ of $P$ consists only of triangles, then we call it a \textit{vertex-edge triangulation} of $P$.
\end{defi}

We make the following useful observation: if $P$ is a triangle, then every intermediate polygon appearing in the vertex-edge subdivision is also a triangle.

\begin{defi} 
    Suppose that $\mathcal{A}$ is a family of non-overlapping convex polygons contained in a polygon $P$. A \textit{vertex-edge refinement} of $\mathcal{A}$ with respect to $P$ is a triangulation $\mathcal{T}$ with the following properties:
    \begin{enumerate}
    \item $\mathcal{T}$ is a vertex-edge triangulation of $P$;
    \item Each polygon $A \in \mathcal{A}$ is the union of triangles in $\mathcal{T}$.
    \item For each $A \in \mathcal{A}$, the set $\mathcal{T}_A = \{T \in \mathcal{T}: T \subset A\}$ is itself a vertex-edge triangulation of $A$. 
\end{enumerate}
\end{defi} 

The point of this definition is that $\mathcal{T}$ is simultaneously a vertex-edge triangulation of the original triangle $T$ and each of the polygons in $\mathcal{A}$. 

\subsection{Existence of vertex-edge refinements}

First, we have a lemma about recognizing vertex-edge triangulations. The situation in this lemma is illustrated in \Cref{fig:recognizing}.
\begin{lemm} \label{lemm:recognize_ve_partition}
    Let $P$ be a convex polygon of perimeter less than $2D_\kappa$ that is a $\CAT(\kappa)$ space. Let $J_i$ ($i=1,\ldots,n-1$) be a set of non-crossing geodesic arcs in $P$ that intersect $\partial P$ only at the endpoints and divide $P$ into $n$ non-overlapping polygons $P_1, \ldots, P_n$. Assume that each $P_i$ is a triangle or quadrilateral whose edges are subarcs of $J_{i-1}$ and $J_i$ (when these are defined) and $\partial P$. Moreover, if $P_i$ is a quadrilateral, then we assume that it does not intersect $P_{i-1}$ and $P_{i+1}$ along adjacent edges of $P_i$. Finally, we require each original vertex of $P$ to be a vertex of one of the $P_i$. We form a triangulation $\mathcal{T}$ of $P$ by dividing each $P_i$ that is a quadrilateral into triangles by connecting two non-adjacent vertices with a geodesic. Then $\mathcal{T}$ is a vertex-edge triangulation of $P$.   
\end{lemm}

\begin{figure} 
    \centering

        \begin{tikzpicture}
        \node () at (0,0) {\includegraphics[height=1.75in]{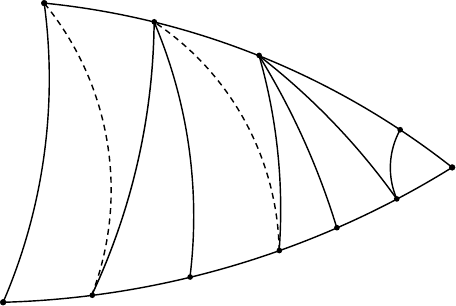}};
        \node () at (-2.5,-1.1) {$P_1$};
        \node () at (-1.1,-1) {$P_2$};
        \node () at (-.1,-.9) {$P_3$};
        \node () at (1.1,-.7) {$P_4$};
        \node () at (1.8,-.5) {$P_5$};
        \node () at (2.1,.2) {$P_6$};
        \node () at (2.7,-.2) {$P_7$};
        \end{tikzpicture}
    \caption{A polygon as in \Cref{lemm:recognize_ve_partition}}
    \label{fig:recognizing}
\end{figure}
\begin{proof}
Observe first that the perimeter assumption on $P$ implies that $P$ has at least three original vertices.

The lemma follows by induction on $n$, the number of subpolygons comprising $P$. It is clear that the conclusion holds if $n=1$. Assume next that the conclusion holds whenever $n \leq N-1$ for some integer $N \geq 2$. Let $P$ be a polygon comprised of $n=N$ subpolygons as in the statement of the lemma. In the first case, suppose that $P$ has an original vertex that is a vertex of a polygon $P_i$ for some $i \in \{2, \ldots, n-1\}$. This vertex must be contained in a common edge of $P_i$ with $P_{i-1}$ or $P_{i+1}$. Subdivide $P$ along this common edge, giving two subpolygons handled by the inductive hypothesis. 

Otherwise, either $P_1$ or $P_n$ is a quadrilateral containing two original vertices of $P$ not in common with $P_2$ or $P_{n-1}$, respectively. Say that $P_1$ has this property. Let $e$ be the edge of $\mathcal{T}$ containing one of these original vertices and crossing through $P$; then $e$ must connect one original vertex of $P$ and one vertex of the common edge with $P_2$. So we subdivide $P$ first along $e$, then along the common edge with $P_2$. This divides $P$ into the two subtriangles of $P_1$ and the polygon $P' = P_2 \cup \cdots \cup P_n$, which is handled by the inductive hypothesis.
\end{proof}

The main result of this section is the following.

\begin{thm} \label{thm:ve_refinement}
If $T$ is a convex triangle of perimeter less than $2D_\kappa$ that is a $\CAT(\kappa)$ space, then any finite family $\mathcal{A}$ of non-overlapping convex polygons contained in $T$ has a vertex-edge refinement with respect to $T$. 
\end{thm}
The proof is somewhat technical, but the geometric idea at its core can be easily read off from \Cref{fig:vertex_edge}
\begin{proof} 
Let $T = \triangle pqr$ be such a triangle and let $\mathcal{A} = \{A_i\}_{i=1}^n$ be a finite collection of non-overlapping convex polygons in $T$. Enumerate the vertices of the $A_i$ as $\{v_j\}_{j=1}^m$. We give the following procedure to find a vertex-edge refinement. Note that our assumption on $T$ implies that any pair of points in $T$ are connected by a unique geodesic.
\begin{enumerate}[leftmargin=4mm]
    \item For each vertex $v_j$ not contained in $[pq]$ or $[pr]$, extend the unique geodesic $[pv_j]$ to a geodesic that intersects the opposite edge $[qr]$ using \Cref{lemm:extend}. See \Cref{fig:initial_geodesics}. This yields a collection of distinct geodesics $\{I_k\}_{k=1}^K$, enumerated in increasing order based upon where the endpoint lies on $[qr]$, oriented from $q$ to $r$. Note that two geodesics $I_{j_1}, I_{j_2}$ cannot share the same endpoint, since otherwise they must coincide by uniqueness of geodesics. Also, set $I_0 = [pq]$ and $I_{K+1} = [pr]$. For all $k \in \{0, \ldots, K\}$, the geodesics $I_k,I_{k+1}$ together with a subarc of $[qr]$ bound a triangle $T_k$.
    \item For each triangle $T_k$ apply the following. Let $\mathcal{A}_k$ denote the collection of polygons of the form $A_i \cap T_k$ intersecting the interior of $T_k$; denote this polygon by $T_k^i$. By \Cref{prop:convex_intersection}, $T_k^i$ is a convex polygon. Note that each vertex of each $T_k^i$ is contained in $I_k \cup I_{k+1}$. This implies that every edge of some $T_k^i$ not contained in $I_k$ or $I_{k+1}$ must connect $I_k$ to $I_{k+1}$. Enumerate such distinct edges as $\{J_k^j\}_{j=1}^{J_k}$ in increasing order from $[qr]$ to the vertex $p$. Also, denote the subarc of $[qr]$ contained in $T_k$ by $J_k^0$. Note that $J_k^j$ and $J_{k}^{j+1}$ along with the relevant subarcs of $I_k$ and $I_{k+1}$ form either a triangle or quadrilateral. We subdivide each triangle $T_k$ as follows. For all $j \in \{0, \ldots, J_k-1\}$, if $J_k^j$ and $J_k^{j+1}$ do not share an endpoint, then connect an endpoint of $J_k^j$ to an endpoint of $J_k^{j+1}$ by a geodesic, forming two triangles. In this manner we obtain a triangulation $\mathcal{T}$ of $T_k$. The final product is shown in \Cref{fig:end_result}.
\end{enumerate}
\begin{figure} 
    \centering
    \subfloat[Initial setup]{
        \begin{tikzpicture}
           \node () at (0,0) {\includegraphics[width = 1.75in]{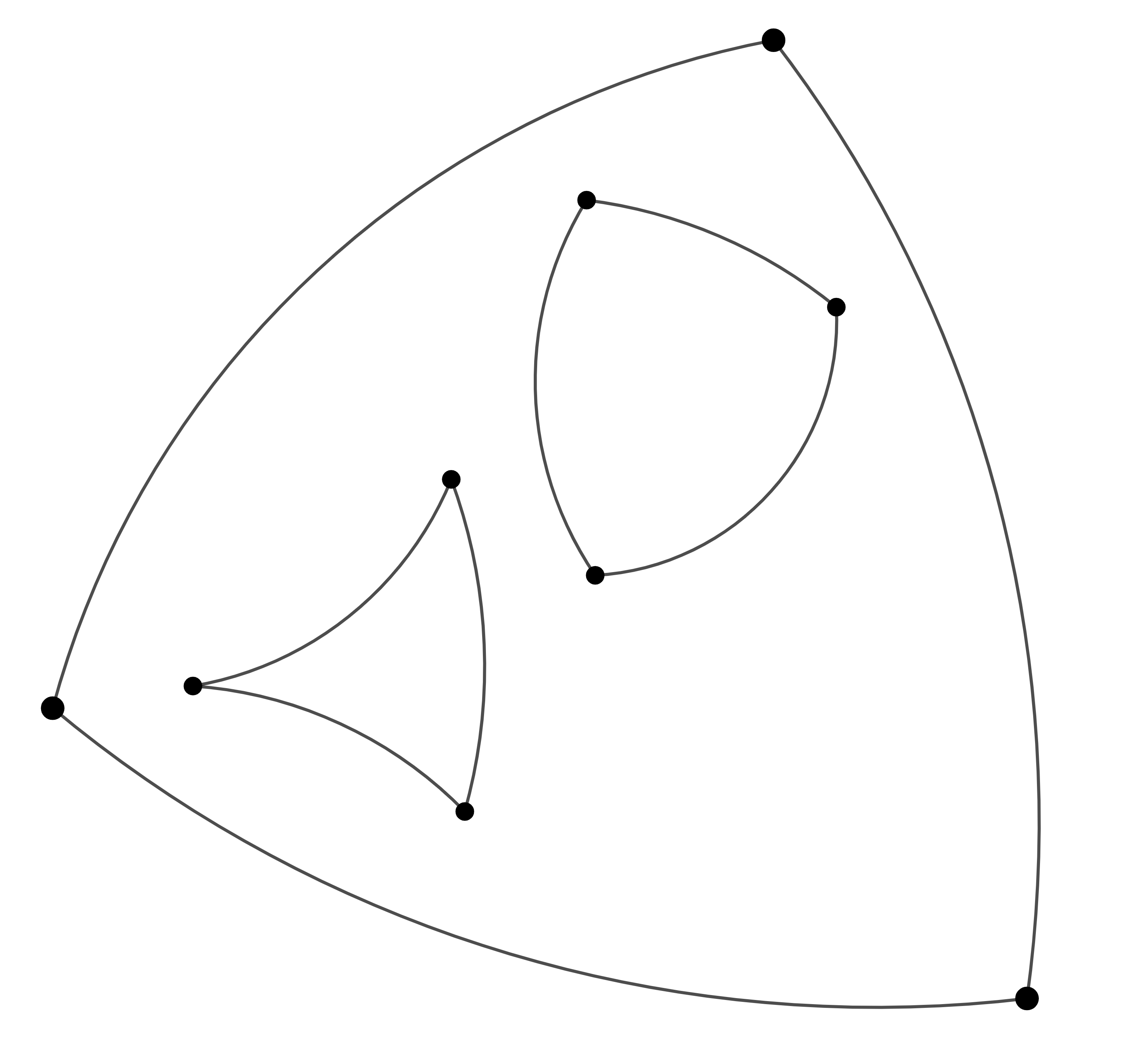}};
           \node () at (-.7,1.6) {{\Large $T$}};
           \node () at (-2.1,-1) {$p$};
           \node () at (1.1,2.0) {$q$};
           \node () at (1.8,-2.1) {$r$};
           \node () at (-.6,-.5) {$A_1$};
           \node () at (.4,.8) {$A_2$};
        \end{tikzpicture}
    \label{fig:setup}}
    \hfill
    \subfloat[Drawing the initial geodesics]{
        \begin{tikzpicture}
           \node () at (0,0) {\includegraphics[width = 1.75in]{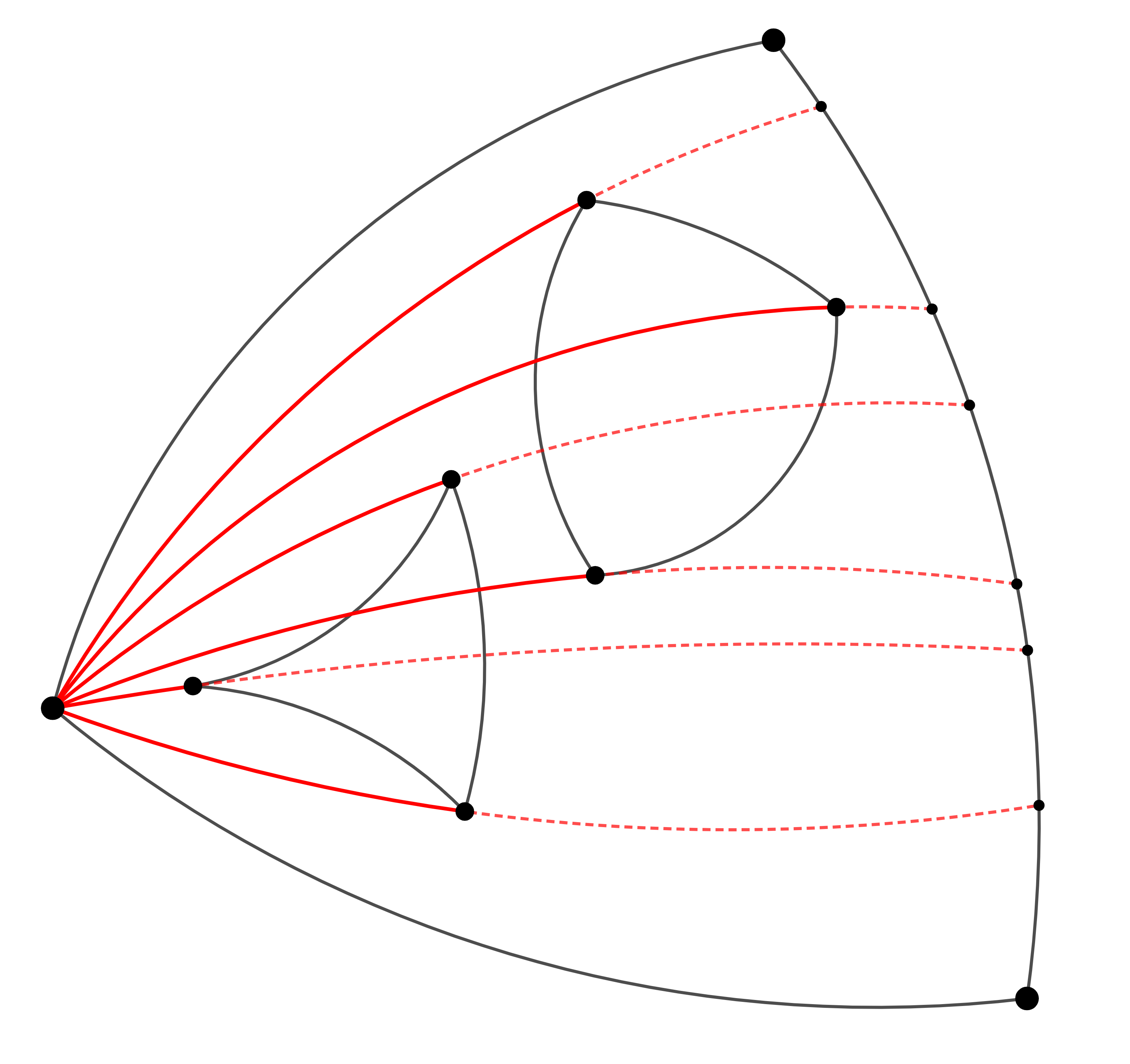}};
           \node () at (-.7,1.6) {{\Large $T$}};
           \node () at (-2.1,-1) {$p$};
           \node () at (1.1,2.0) {$q$};
           \node () at (1.8,-2.1) {$r$};
        \end{tikzpicture}
    \label{fig:initial_geodesics}}
    \hfill
    \subfloat[The final triangulation]{ 
        \begin{tikzpicture}
           \node () at (0,0) {\includegraphics[width = 1.75in]{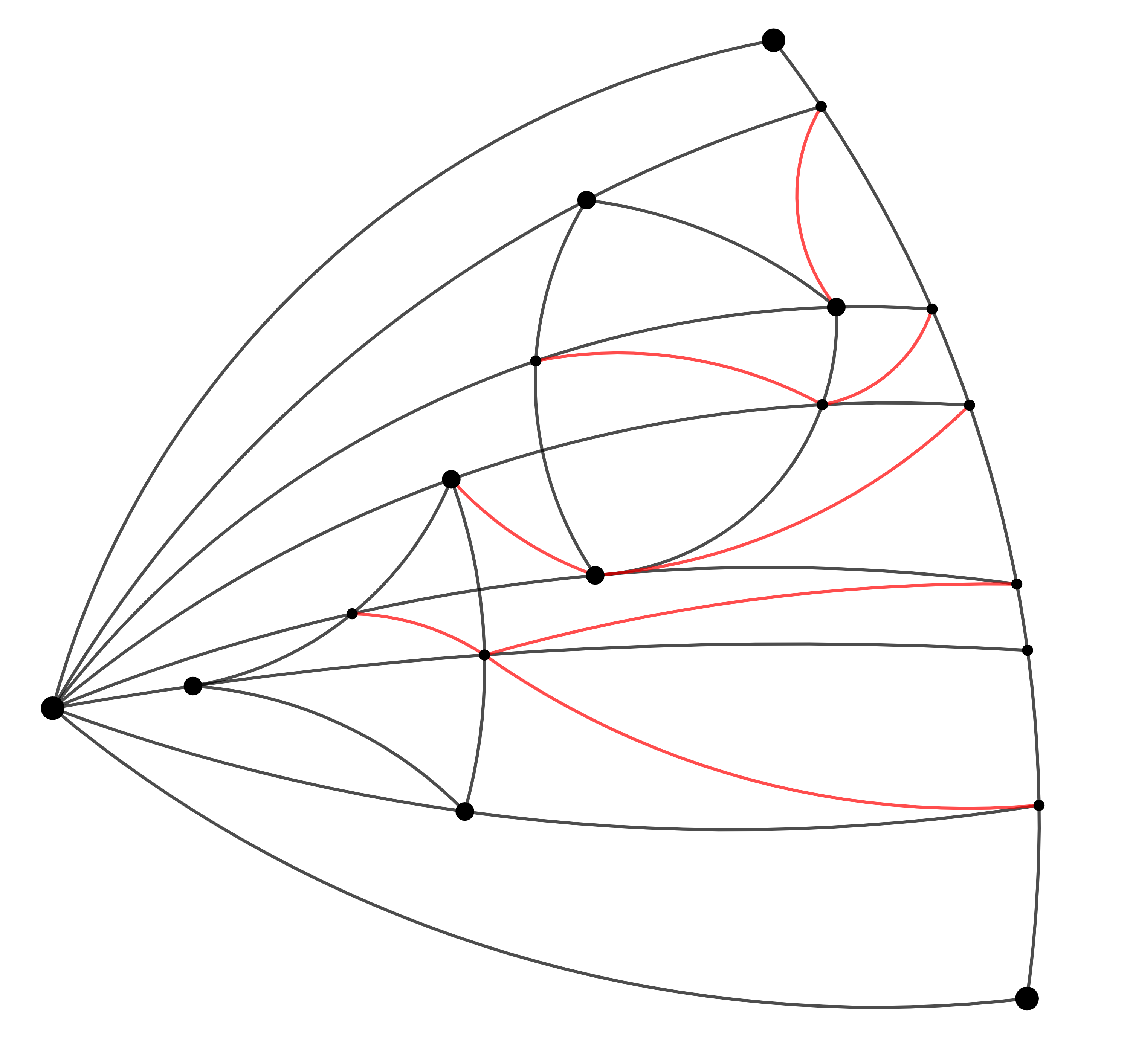}};
           \node () at (-.7,1.6) {{\Large $T$}};
           \node () at (-2.1,-1) {$p$};
           \node () at (1.1,2.0) {$q$};
           \node () at (1.8,-2.1) {$r$};
        \end{tikzpicture}
    \label{fig:end_result}}
    \caption{}
    \label{fig:vertex_edge}
\end{figure}

It follows from construction that $\mathcal{T}$ is a vertex-edge triangulation of $T$. Indeed, $\mathcal{T}$ is obtained first by drawing geodesics from $p$ to the opposite edge to form the subtriangles $T_k$. Then each $T_k$ is the union of triangles satisfying the requirements of \Cref{lemm:recognize_ve_partition}. We conclude that $\mathcal{T}$ is vertex-edge. 


Finally, we claim that, for each $A \in \mathcal{A}$, the collection $\mathcal{T}_A$ can be obtained by vertex-edge subdivision of $A$. Observe that the geodesics $\{I_k\}_{k=1}^K$ divide $A$ into the non-overlapping union of triangles and quadrilaterals satisfying the requirements of \Cref{lemm:recognize_ve_partition}. We now apply this lemma to conclude that $\mathcal{T}_A$ is a vertex-edge triangulation of $A$.
\end{proof} 

\subsection{Subdividing triangles}

The main usefulness of vertex-edge triangulations stems from the following lemma, which states that both excess and model area behave monotonically when subdividing triangles.

\begin{lemm}
\label{lemm:vertex-edge-prelim}
Let $T=\Delta pqr$ be a triangle of perimeter at most $2D_\kappa$ satisfying the $\CAT(\kappa)$ condition and let $s$ be a point on the side $[qr]$. Let $\mathcal{T}$ be the vertex-edge triangulation consisting of the triangles $T_1=\Delta psq$ and $T_2=\Delta psr$ formed by connecting $p$ and $s$. Then
\begin{enumerate}
    \item \label{item:total_excess} Total excess does not decrease: \[\delta(T)\leq\delta(T_1)+\delta(T_2)\]
    \item \label{item:total_model_area} Total model area does not increase: 
    \[|\bar{T}_1|+|\bar{T}_2|\leq |\bar{T}|,\] 
    where $\bar{T},\bar{T}_1,\bar{T}_2$ are the corresponding comparison triangles in the model space $M_\kappa^2$. 
\end{enumerate}
\end{lemm}
We note that statement \eqref{item:total_excess} is standard; see Lemma III.9 in \cite{AZ:67} for an example of a more general statement. 
\begin{proof}
Denote the angles of $T$ by $\alpha, \beta, \gamma$ and the angles of each $T_i$ by $\alpha_i, \beta_i, \gamma_i$, where $\alpha$, $\alpha_1$, $\alpha_2$ are the angles at the shared vertex $p$ and $\gamma_1$, $\beta_2$ are the angles at the vertex $s$. 

We first prove \eqref{item:total_excess}. The triangle inequality for angles implies that $\alpha \leq \alpha_1+\alpha_2$ and $\pi\leq \gamma_1+\beta_2$. 
    It follows that
    \begin{align*}
        \delta(T)&=\alpha+\beta + \gamma-\pi \\
        &\leq \alpha_1 + \alpha_2 + \beta + \gamma + (\gamma_1+\beta_2-\pi)-\pi\\
        &=(\alpha_1+\gamma_1+\beta - \pi)+(\alpha_2+\beta_2+\gamma-\pi)\\
        &=\delta(T_1)+\delta(T_2).
    \end{align*}


Next, to prove \eqref{item:total_model_area},
    the main task is to show that the comparison triangles $\bar{T}_1$ and $\bar{T}_2$ can be fitted into $\bar{T}$ such that $\bar{T}_1$ and $\bar{T}_2$ are non-overlapping. Our argument is based on Alexandrov's Lemma; see \cite[I.2.16]{BH:99}. Consider the triangle $\bar{T}_1$ arranged to share the side $[\bar{p}\bar{q}] \subset \bar{T}$ and have overlapping interior. Similarly, the triangle $\bar{T}_2$ is arranged to share the side $[\bar{p}\bar{r}] \subset \bar{T}$ and have overlapping interior. Let $\bar{y}$ denote the third vertex of $\bar{T}_1$ and let $\bar{z}$ denote the third vertex of $\bar{T}_2$. The situation is summarized in \Cref{fig:subdivision}.
    

\begin{figure} 
    \centering
        \begin{tikzpicture}
           \node () at (0,0) {\includegraphics[width = 2.75in]{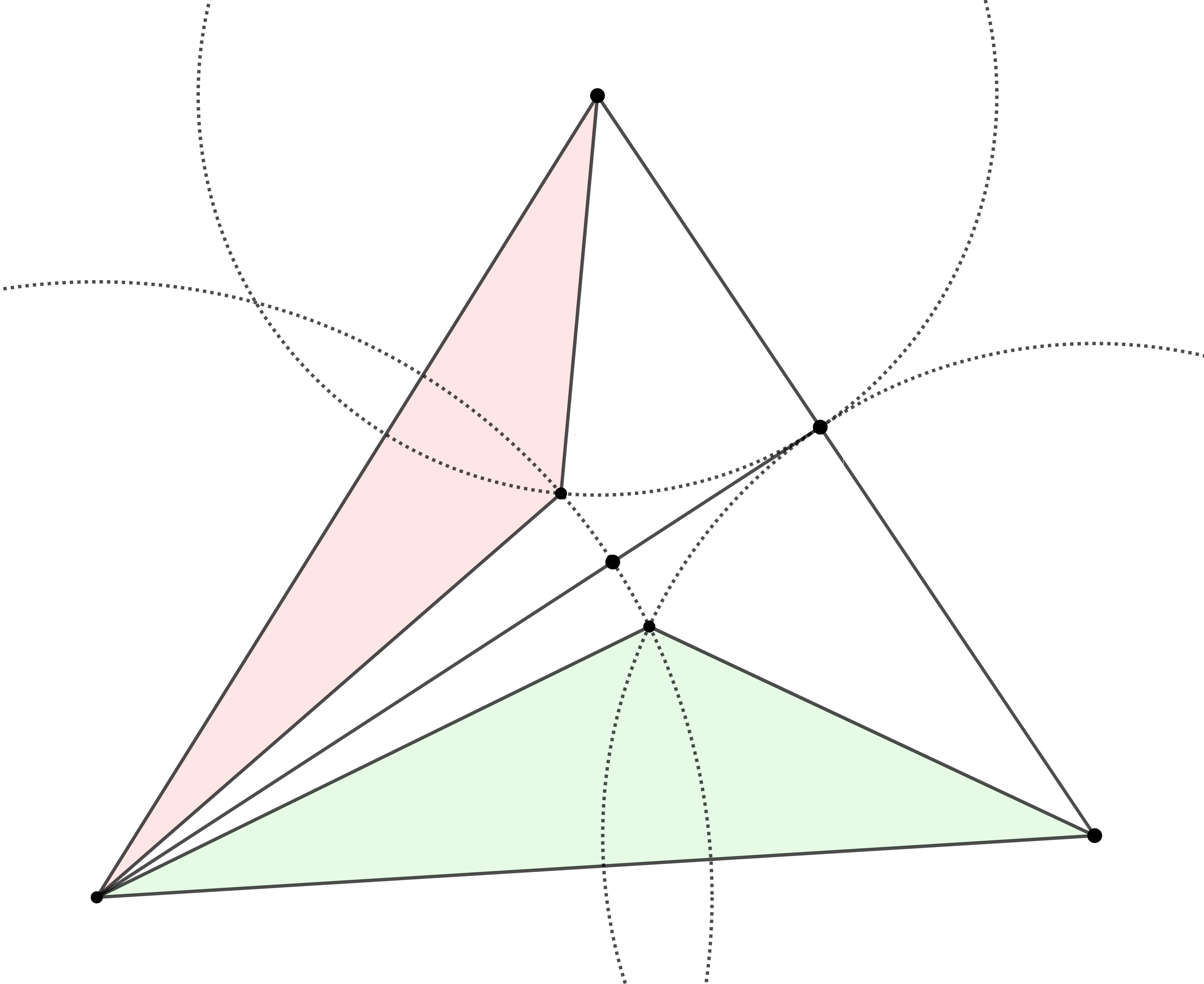}};
           \node () at (1.1,1.3) {{\Large $\bar{T}$}};
           \node () at (-1.2,-.2) {{\Large $\bar{T}_1$}};
           \node () at (-.5,-1.8) {{\Large $\bar{T}_2$}};
           \node () at (-2.9,-2.6) {$\bar{p}$};
           \node () at (0,2.5) {$\bar{q}$};
           \node () at (2.8,-2.2) {$\bar{r}$};
           \node () at (.5,-.7) {$\bar{z}$};
           \node () at (0,.2) {$\bar{y}$};
        \end{tikzpicture}
    \caption{Subdividing a $\CAT(\kappa)$ triangle}
    \label{fig:subdivision}
\end{figure}

Let $\bar{\alpha}$, $\bar{\beta}$, $\bar{\gamma}$ denote the angle in $\bar{T}$ corresponding to $\alpha$, $\beta$, $\gamma$, respectively, and likewise for each $\bar{T}_i$. The $\CAT(\kappa)$ condition implies that $\gamma_1 \leq \bar{\gamma}_1$ and $\beta_2 \leq \bar{\beta}_2$, and hence $\pi \leq \bar{\gamma}_1 + \bar{\beta}_2$. Then Alexandrov's Lemma  states that $\bar{\alpha} \geq \bar{\alpha}_1 + \bar{\alpha}_2$, $\bar{\beta}_1 \leq \bar{\beta}$ and $\bar{\gamma}_2 \leq \bar{\gamma}$.

Since $\bar{\beta}_1 \leq \bar{\beta}$ and $\alpha_1 \leq \alpha$, we see that $\bar{y}$ is contained in $\bar{T}$. Likewise, $\bar{z}$ is also contained in $\bar{T}$. This implies that the triangles $\bar{T}_1$ and $\bar{T}_2$ are entirely contained in $\bar{T}$. Moreover, since $\bar{\alpha} \geq \bar{\alpha}_1 + \bar{\alpha}_2$, the segments $[\bar{p}\bar{r}]$, $[\bar{p}\bar{z}]$, $[\bar{p}\bar{y}]$, $[\bar{p}\bar{q}]$ emanate from $\bar{p}$ counterclockwise in this order, with possibly $[\bar{p}\bar{z}]$, $[\bar{p}\bar{y}]$ coinciding. From this it follows that $\bar{T}_1$ and $\bar{T}_2$ do not overlap. We conclude that $|\bar{T}_1| + |\bar{T}_2| \leq |\bar{T}|$.
\end{proof}

\section{$\CAT(\kappa)$ surfaces have bounded curvature} \label{sec:bic_proofs}

    We now proceed with the proof of \Cref{thm:cat_implies_bic}. Let $X$ be a $\CAT(\kappa)$ surface and let $x \in X$. The first step is to find a neighborhood of $x$ for which we can verify the condition of bounded curvature.  Let $r \in (0,D_\kappa/2)$ be sufficiently small so that $B(x,r)$ is a $\CAT(\kappa)$ space. Since $r< D_\kappa/2$, the set $B(x,r)$ is necessarily convex and contractible in $X$, and in particular is a topological open disk. We assume without loss of generality that $\kappa = 1$. Use \Cref{thm:triangulation} to find a triangulation $\mathcal{S}$ of $B(x,r)$ such that each $T \in \mathcal{S}$ satisfies $d(T,\partial B(x,r)) > \text{diam}(T)$. Let $S_1, \ldots, S_n$ be the collection of triangles in $\mathcal{S}$ that contain $x$. Then $S = \bigcup_{i=1}^n S_i$ is a closed topological disk such that $x \in S^\circ$. To see this, observe from the convexity of each triangle and uniqueness of geodesics that any two distinct triangles $S_i$, $S_j$ intersect in a possibly degenerate arc containing the element $x$. We claim that $S^\circ$ is the neighborhood of $x$ as required by the definition of surface of bounded curvature. 
    
    Let $\mathcal{T} = \{T_j\}_{j=1}^m$ be an arbitrary collection of simple non-overlapping triangles in $S^\circ$. Our goal is to find a uniform upper bound on $\delta(\mathcal{T})$. We apply \Cref{thm:ve_refinement} to find for each $T_j \in \mathcal{T}$ a triangulation $\mathcal{T}_j$ of $T_j$ that is a vertex-edge refinement of the collection of polygons $S_i \cap T_j$, $i = 1, \ldots, n$, intersecting the interior of $T_j$ with respect to $T_j$. Note that \Cref{prop:convex_intersection} implies that each set $S_i \cap T_j$ is in fact a polygon provided that its interior is non-empty. Observe from \Cref{lemm:vertex-edge-prelim} that $\delta(T_j) \leq \delta(\mathcal{T}_j)$. 
    
    Next, for each triangle $S_i$, let $\mathcal{S}_i$ be the set of triangles in $\bigcup_j \mathcal{T}_j$ contained in $S_i$. We apply \Cref{thm:ve_refinement} a second time to $\mathcal{S}_i$ to find a vertex-edge refinement of $\mathcal{S}_i$ with respect to $S_i$. Denote this by $\mathcal{S}_i'$. Suppose that $T \in \mathcal{T}_j$ and contained in $S_i$. Let $\mathcal{S}_i'(T)$ denote the subset of $\mathcal{S}_i'$ contained in $T$. Then, since $\mathcal{S}_i'(T)$ is a vertex-edge triangulation of $T$, we have hence $\delta(T) \leq \delta(\mathcal{S}_i'(T))$ by \Cref{lemm:vertex-edge-prelim}. This gives
    \begin{align*}
        \sum_{j=1}^{m}\delta(T_j)& \leq \sum_{j=1}^{m} \sum_{T \in \mathcal{T}_j} \delta(T) \leq \sum_{i=1}^{n} \sum_{T \in \mathcal{S}_i} \sum_{T' \in \mathcal{S}_i'(T)} \delta(T') \leq \sum_{i=1}^{n} \sum_{T \in \mathcal{S}_i} \sum_{T' \in \mathcal{S}_i'(T)} |\bar{T}_1'|.
    \end{align*}
    Here, $\bar{T}_1'$ is the model triangle in $M_1^2$ corresponding to $T'$. Similarly, for each triangle $S_i$ we let $\bar{S}_i$ denote the model triangle of $S_i$ in $M_1^2$. Since $\mathcal{S}_i'$ is a vertex-edge refinement with respect to $S_i$, \Cref{lemm:vertex-edge-prelim} gives
    \[\sum_{i=1}^{n} \sum_{T \in \mathcal{S}_i} \sum_{T' \in \mathcal{S}_i'(T)} |\bar{T}_1'| \leq \sum_{i=1}^{n} \sum_{T' \in \mathcal{S}_i'} |\bar{T}_1'| \leq \sum_{i=1}^{n} |\bar{S}_i| .\] 
    Combining these inequalities gives an upper bound for $\delta(\mathcal{T}) = \sum_{j=1}^{m}\delta(T_j)$ independent of the choice of $\mathcal{T}$. This completes the proof.

\section{Excess-curvature inequality} \label{sec:excess_curvature}

In this section, we give a proof of \Cref{thm:excess_curvature} relating the excess and total curvature of an arbitrary Jordan triangle in a $\CAT(\kappa)$ surface. First we verify \Cref{prop:hausdorff_vs_alexandrov}, that the Hausdorff area of a surface $X$ of bounded curvature without cusp points coincides with the Alexandrov area (see \Cref{sec:surfaces_bounded_curvature}). This is an application of a theorem of Burago \cite{Bur:65} that the metric tangent at each point is a cone over a circle; see also Theorem 9.10 in \cite{Res:93}. Recall that, for a Borel set $A \subset X$, we use $|A|$ to denote the Hausdorff area (i.e., $2$-measure) and $\text{Area}(A)$ to denote Alexandrov area.  

\begin{proof}[Proof of \Cref{prop:hausdorff_vs_alexandrov}]
We must show that $|A| = \text{Area}(A)$ for an arbitrary Borel set $A$. Since both Hausdorff area and Alexandrov area are determined by their values on polygons, without loss of generality we may assume that $A$ is a polygon. Let $A \subset X$ be a polygon. We can consider $A$ as its own metric space with the induced length metric; this does not affect either notion of area. Let $\varepsilon>0$ be given. For each $x \in X$, let $B_x$ be an open metric ball at $x$ with the property of being $(1+\varepsilon)$-bi-Lipschitz equivalent to a ball in a Euclidean cone over a circle, and let $\mathcal{U} = \{B_x: x \in A\}$. Let $\mathcal{T}$ be a triangulation of $A$ subordinate to the cover $\mathcal{U}$ and having mesh $\varepsilon$ given by \Cref{thm:triangulation}. Define a polyhedral surface $A_\varepsilon$ by replacing each triangle $T \in \mathcal{T}$ with the corresponding model triangle $\bar{T}_0$. It is clear that $|A_\varepsilon| = \text{Area}(A_\varepsilon)$. Moreover, we have that each triangle $T \in \mathcal{T}$ is $(1+\varepsilon)$-bi-Lipschitz equivalent to the triangle $\bar{T}_0$, which implies that  
\[(1+\varepsilon)^{-2}|A| \leq |A_\varepsilon| \leq (1+\varepsilon)^2|A|.  \] 
Letting $\varepsilon \to 0$, we have that $\text{Area}(A_\varepsilon)$ converges to $\text{Area}(A)$; see Theorem VIII.2 in \cite{AZ:67}. From this we conclude that $|A| = \text{Area}(A)$, which completes the proof.     
\end{proof}

We now proceed with the proof of \Cref{thm:excess_curvature}. 

\begin{proof}[Proof of \Cref{thm:excess_curvature}]
Let $X$ be a $\CAT(\kappa)$ surface and let $T \subset X$ be a Jordan triangle with angles $\alpha , \beta, \gamma$. We may consider $T$ as its own subspace of $X$, having the induced length metric $d_T$. Observe that $d_T(x,y) \geq d(x,y)$ for all $x,y \in T$, with equality if $x,y$ belong to the same edge of $T$. We conclude that passing to the metric $d_T$ does not decrease the angles $\alpha, \beta, \gamma$ and hence does not decrease the excess $\delta(T)$. 

We now construct a sequence of model polyhedral surfaces $T_n$ converging uniformly to $T$, where each face is a triangle of constant curvature $\kappa$ of diameter at most $\varepsilon_n$. Note that such a sequence can be found by invoking the general theory of surfaces of bounded curvature. See Theorem III.10 of \cite{AZ:67} for the $\kappa=0$ case; the generalization to arbitrary $\kappa$ is a simple modification. However, for completeness, we provide details for this step.

Fix a decreasing sequence $(\varepsilon_n)$ of positive real numbers limiting to $0$. Assume further that $\varepsilon_1$ is sufficiently small that for each point $x \in T$ the ball $B(x,2\varepsilon_1)$ is $\CAT(\kappa)$. Apply \Cref{thm:triangulation} (with arbitrary choice of open cover $\mathcal{U}$) to obtain a sequence of triangulations $\mathcal{T}_n$ of $T$ corresponding to the parameters $\varepsilon_n$. 

Let $\mathcal{E}_n$ be the edge graph of $\mathcal{T}_n$. Observe that $d(x,y) \leq \widetilde{d}_n(x,y)$ and that, by \Cref{thm:triangulation}, the induced length metric $\widetilde{d}_n$ on $\mathcal{E}_n$ satisfies
\[\widetilde{d}_n(x,y) \leq d(x,y) +\varepsilon_n\] 
for all $x,y \in \mathcal{E}_n$. Define the surface $X_n$ by, for each $S \in \mathcal{T}_n$, gluing in a model triangle of constant curvature $\kappa$ having the same edge lengths into the edge graph $\mathcal{E}_n$. Identify $\mathcal{E}_n$ with a subset of $X_n$ in the natural way, and $X_n$ with $X$ by choosing some bijection that is the identity map on the edge graph $\mathcal{E}_n$. Having made these identifications, denote the new metric on $X$ by $\bar{d}_n$. Let $S_n^j$, $1 \leq j \leq k_n$, denote the collection of faces of $T_n$, each of which is a triangle of constant curvature $\kappa$.

It follows from the $\CAT(\kappa)$ condition that $d(x,y) \leq \bar{d}_n(x,y)$ for all $x,y \in \partial S$ for each triangle $S \in \mathcal{T}_n$. Given two points $x,y \in X$, let $\bar{\Gamma}_n$ be a curve from $x$ to $y$ whose $\bar{d}_n$-length satisfies $\ell_{\bar{d}_n}(\Gamma_n) \leq \bar{d}_n(x,y) + \varepsilon_n$. We assume that the restriction of $\bar{\Gamma}_n$ to each triangle $S \in \mathcal{T}_n$ is a $\bar{d}_n$-geodesic. Define the curve $\Gamma$ by replacing the subcurve of $\bar{\Gamma}_n$ intersecting a given triangle $T$ with the $d$-geodesic connecting the same endpoints. Then $\Gamma$ satisfies $\ell_d(\Gamma) \leq \ell_{\bar{d}_n}(\bar{\Gamma}_n) + 2\varepsilon_n$. Combining these inequalities gives
\[d(x,y) \leq \ell_d(\Gamma) \leq \ell_{\bar{d}_n}(\bar{\Gamma}_n) + 2\varepsilon_n \leq \bar{d}_n(x,y) + 3\varepsilon_n .\] 

Next, for any points $x,y \in X$, we can find points $x',y' \in \mathcal{E}_n$ such that $\bar{d}_n(x,x') \leq \varepsilon_n$ and $\bar{d}_n(y,y') \leq \varepsilon_n$. So for all $x,y \in X$, we have
\[\bar{d}_n(x,y) \leq \bar{d}_n(x',y') + 2\varepsilon_n \leq \widetilde{d}_n(x',y')+2\varepsilon_n \leq d(x',y')+ 3 \varepsilon_n \leq d(x,y) + 5\varepsilon_n. \]
We conclude that the sequence of surfaces $T_n$ converges uniformly to $T$ with its induced length metric.

Note that $T_n$ is also a triangle, since the edges of $T$ are still geodesics in the polyhedral approximation. Let $\alpha_n, \beta_n, \gamma_n$ denote the angles in $T_n$ corresponding to the same edges in $T$ as $\alpha,\beta,\gamma$, respectively. The $\CAT(\kappa)$ condition implies that $\alpha \leq \alpha_n$, $\beta \leq \beta_n$, $\gamma \leq \gamma_n$, and thus that $\delta(T) \leq \delta(T_n)$. The excess moreover satisfies \[\delta(T_n) \leq \sum_{j=1}^{k_n} \delta(S_n^j);\] see Lemma III.9 in \cite{AZ:67}, noting that the proof applies to our situation.

We observe that $\delta(S_n^j) = \kappa|S_n^j|$ for each triangle $S_n^j$ by Girard's Theorem and its generalization. Combining these inequalities gives
\[\delta(T) \leq \delta(T_n) \leq \sum_{j=1}^{k_n} \delta(S_n^j) = \sum_{j=1}^{k_n} \kappa |S_n^j| = \kappa |T_n|.\]
Letting $n \to \infty$, the right-hand side converges to $|T|$. We conclude that $\delta(T) \leq \kappa |T|$.
\end{proof}

\section{Approximation by smooth Riemannian surfaces} \label{sec:smooth_approximation}

In this section, we prove that every $\CAT(\kappa)$ surface, $\kappa  \in \mathbb{R}$, is the uniform limit of smooth Riemannian surfaces with Gaussian curvature bounded above by $\kappa$ and having locally uniformly bounded area. As already discussed in the proof of \Cref{thm:excess_curvature}, it is well-known that such a surface is the uniform limit of $\CAT(\kappa)$ model polyhedral surfaces. 
What remains is to show that one can smoothen out the vertices while retaining the $\CAT(\kappa)$ condition.

Let $X$ be a polyhedron with faces of constant curvature $\kappa$. The metric in a neighborhood of a vertex, in polar coordinates, is defined by the conformal length element $ds^2 = \lambda(r)(dr^2+r^2 d\theta^2)$, where
\begin{equation} \label{equ:polar_metric}
  \lambda(r) = \frac{4\alpha^2 r^{2(\alpha-1)}}{(1+\kappa \cdot r^{2\alpha})^2}  
\end{equation}
and $\alpha\geq 1$ is the total angle at the vertex divided by $2\pi$. If $\alpha = 1$, then we just have the usual constant-curvature metric. Observe that $\lim_{r\to 0} \lambda(r)=0$ if $\alpha>1$. The Gaussian curvature of this metric outside of the vertex is given by the formula
\begin{equation*} 
  K(r) = \frac{-1}{2\lambda(r)}\left(\frac{\partial^2}{\partial r^2} \log(\lambda(r)) + \frac{1}{r} \frac{\partial}{\partial r} \log(\lambda(r)) \right).  
\end{equation*}

We have the following main lemma.  

\begin{lemm} \label{lemm:polyhedral_approx}
Consider the Euclidean ball $B_{\Euc}(\mathbf{0},R) \subset \mathbb{R}^2$ for some $R>0$, equipped with the metric $ds^2 = \lambda(r)(dr^2+r^2 d\theta^2)$, with $\lambda$ given by \eqref{equ:polar_metric}, with $\alpha > 1$. For all $\delta \in (0,R)$ sufficiently small, there is a smooth Riemannian metric $ds_\delta^2 = \lambda_\delta(r)(dr^2+r^2d\theta^2)$ that agrees with $ds^2$ on $B_{\Euc}(\mathbf{0},R)$ outside the Euclidean ball $B_{\Euc}(\mathbf{0},\delta)$ and has Gaussian curvature at most $\kappa$. Moreover, $\lambda_\delta$ is an increasing function on $B_{\Euc}(\mathbf{0},\delta)$. 
\end{lemm}
\begin{proof} 
We give two approaches to proving this lemma: the first based on interpolating with a flat surface and the second based on interpolating with a hyperbolic surface. The first approach is simpler but only covers the case where $\kappa \geq 0$. 

\textit{1. Interpolating with a flat metric.} We assume here that $\kappa \geq 0$. We define the smoothened metric $\lambda_\delta$ as follows. First, we define for each $0<a<b$ the smooth cutoff function $\varphi_{a,b}: [0, \infty) \to [0,1]$ by
\[\varphi_{a,b}(r) = \left\{\begin{array}{cl} 0 & \text{ if }  r \leq a\\ \frac{e^{-1/(r-a)}}{e^{-1/(b-r)} +e^{-1/(r-a)}} & \text{ if } a \leq r \leq b\\ 1 &  \text{ if } r \geq b \end{array}  \right. . \] 
Write $\varphi_b = \varphi_{b/2,b}$. Next, let 
\[g_\delta(r) = \left(\frac{\partial}{\partial r} \log(\lambda(r))\right) \varphi_{\delta}(r) \] and
\[\lambda_\delta(r) = \lambda(\delta) \exp \left( \int_\delta^r g_\delta(t)\,dt \right). \] 
We claim that this choice of $\lambda_\delta$ gives the required properties. It is easy to see that $\lambda_\delta(r) = \lambda(r)$ whenever $r \geq \delta$. Moreover, $g_\delta(r) = 0$ if $r \leq \delta/2$, from which it follows that $\lambda_\delta$ is constant on the interval $[0,\delta]$. From this, we see that $ds_\delta$ is locally isometric to the plane on $B(\mathbf{0},\delta)$. We conclude that $ds_\delta$ is a smooth Riemannian metric. Observe now that
\begin{equation} \label{equ:log_deriv}
    \frac{\partial}{\partial r} \log(\lambda(r)) =  \frac{-2+2\alpha - 2\kappa \cdot r^{2\alpha}(1+\alpha)}{r(1+\kappa \cdot r^{2\alpha})},
\end{equation}
which is positive provided that $\alpha>0$ and $r$ is sufficiently small depending on $\alpha$. Take $\delta>0$ small enough so that the quantity in \eqref{equ:log_deriv} is positive for all $r \in (0,\delta)$. It follows that $g_\delta(r)$ is non-negative for all $r \in (0,\delta)$ and thus that $\lambda_\delta(r)$ is increasing on this interval. We also choose $\delta$ small enough so that $\lambda(\delta) \leq 1$. 

The curvature of the metric $ds_\delta^2$ is
\begin{align*}
  K_\delta(r) & =  \frac{-1}{2\lambda_\delta(r)}\left(g_\delta'(r) + \frac{g_\delta(r)}{r} \right)  \\
  & = \frac{-1}{2\lambda_\delta(r)}\left(\left(\frac{\partial}{\partial r} \log(\lambda(r))\right) \varphi_\delta'(r) + \left(\frac{\partial^2}{\partial r^2} \log(\lambda(r))\right) \varphi_\delta(r) + \frac{\left(\frac{\partial}{\partial r} \log(\lambda(r))\right) \varphi_\delta(r)}{r} \right).
\end{align*}

To complete the proof, we must show that $K_\delta(r) \leq \kappa$. This is equivalent to the inequality
\begin{equation} \label{equ:curvature_inequality}
  -g_\delta'(r) - \frac{g_\delta(r)}{r} \leq 2\kappa \cdot \lambda_\delta(r),  
\end{equation}
which we can write as
\[-\left(\frac{\partial}{\partial r} \log(\lambda(r))\right) \varphi_\delta'(r) - \left(\frac{\partial^2}{\partial r^2} \log(\lambda(r))\right) \varphi_\delta(r) - \frac{\left(\frac{\partial}{\partial r} \log(\lambda(r))\right) \varphi_\delta(r)}{r} \leq 2\kappa \cdot \lambda(\delta) \exp \left( \int_\delta^r g_\delta(t)\,dt \right). \] 
To handle the first term on the left-hand side of this equation, observe that $\left(\frac{\partial}{\partial r} \log(\lambda(r))\right) \varphi_\delta'(r) \geq 0$, again using the property that $\frac{\partial}{\partial r} \log(\lambda(r))>0$ for all $r \in (0,\delta)$. The other two terms can be written as
\begin{align*}
  \left(-\frac{\partial^2}{\partial r^2} \log(\lambda(r)) - \frac{\left(\frac{\partial}{\partial r} \log(\lambda(r))\right) }{r} \right) \varphi_\delta(r) = 2\kappa \cdot \lambda(r) \varphi_\delta(r).
\end{align*}
Thus, to verify \eqref{equ:curvature_inequality}, it is enough to show that $2\kappa \cdot \lambda(r) \varphi_\delta(r) \leq 2\kappa \cdot \lambda_\delta(r)$ for all $r \in (\delta/2, \delta)$. This is trivially satisfied if $\kappa=0$. If $\kappa>0$, this can be written as
\begin{equation} \label{equ:curvature_verification} \lambda(r) \varphi_\delta(r) \leq \lambda(\delta) \exp \left( \int_\delta^r g_\delta(t)\,dt \right). 
\end{equation} 
First, since $\lambda(r) \leq \lambda(\delta) \leq 1$ and $\varphi_\delta(r) \leq 1$ for all $r \in (0,\delta)$, we have
\[\lambda(r) \varphi_\delta(r) \leq \lambda(r) \leq \lambda(r)^{\varphi_\delta(r)} = \exp \left( (\log \lambda(r) ) \varphi_\delta(r) \right). \]
Then, using the properties that $r< \delta$ and that $\varphi_\delta$ is increasing and bounded by $1$, we obtain
\begin{align*}
    \exp \left( (\log \lambda(r) ) \varphi_\delta(r) \right) & = \exp \left( \left(\log \lambda(\delta) + \int_\delta^r \frac{\partial}{\partial t} \log \lambda(t)\,dt \right) \varphi_\delta(r) \right) \\
    & \leq \exp \left( \left(\log \lambda(\delta) + \int_\delta^r \frac{\partial}{\partial t} \log \lambda(t) \varphi_\delta(t) \,dt \right)  \right) \\
    & = \lambda(\delta) \exp\left(\int_\delta^r g_\delta(t)\,dt \right).
\end{align*}
This concludes our verification that $K_\delta(r) \leq \kappa$. 

\textit{2. Interpolating with the hyperbolic metric.} The following proof handles the general case. It is similar to the previous one, so we focus on the differences. By rescaling, we may assume that $\kappa \geq -1$. Let $\bar{\lambda}(r) = 4/(1-r^2)^2$ be the conformal weight for the hyperbolic metric, which is defined for all $r \in (0,1)$. We make two immediate observations. First, since $\alpha >1$ and $\kappa \geq -1$, we have $\bar{\lambda}(r) \geq \lambda(r)$ for all $r \in (0,1)$. Second, the inequality
\begin{equation} \label{equ:logarithm}
\frac{\partial}{\partial r} \log(\lambda(r)) \geq \frac{\partial}{\partial r} \log(\bar{\lambda}(r)) \end{equation}  
holds for all $r>0$ sufficiently small. Assume then that $\delta>0$ is small enough that \eqref{equ:logarithm} holds for all $r \leq \delta$. We also necessarily assume that $\delta < 1$.

Our main modification is to redefine $g_\delta$ above using the formula
\[g_\delta(r) =  \left(\frac{\partial}{\partial r} \log(\lambda(r))\right) \varphi_{\delta}(r) + \left(\frac{\partial}{\partial r} \log(\bar{\lambda}(r))\right)(1- \varphi_{\delta}(r)).\] 
Keep the same definition for $\lambda_\delta$, except that we use the new formula for $g_\delta$. The conformal weight $\lambda_\delta$ on $B_{\text{Euc}}(\mathbf{0},\delta/2)$ is a rescaling of the hyperbolic metric by the factor $\lambda(\delta) \exp\left(\int_{\delta}^{\delta/2} g_\delta(t)\,dt\right) \leq \lambda(\delta/2)$, which is less than $1$ since $\delta<1$. It follows that $K_\delta(r) \leq \kappa$ whenever $r \leq \delta/2$.

Next, to verify \eqref{equ:curvature_inequality}, there are two details to check. First, we observe from \eqref{equ:logarithm} that  \[\left(\frac{\partial}{\partial r} \log(\lambda(r))- \frac{\partial}{\partial r} \log(\bar{\lambda}(r))\right) \varphi_\delta'(r) \geq 0\] for all $r \in (0, \delta)$. Then, instead of \eqref{equ:curvature_verification} we must check that
\[ \kappa \cdot \lambda(r) \varphi_\delta(r) - \bar{\lambda}(r) (1-\varphi_\delta(r)) \leq \kappa \cdot \lambda_\delta(r)\]
for all $r \in (0,\delta)$. This is trivial if $\kappa=0$. If $\kappa>0$, it suffices to show that
\[ \lambda(r) \varphi_\delta(r) \leq \lambda_\delta(r) ,\]
which can be done similarly to the first part of the proof. In the case where $\kappa<0$, it suffices to show that
\[\lambda(r) \varphi_\delta(r) + \bar{\lambda}(r) (1-\varphi_\delta(r)) \geq \lambda(\delta) \exp \left( \int_\delta^r g_\delta(t) \right) . \]
Note that, since $\kappa<0$, this inequality is reversed relative to \eqref{equ:curvature_verification}. We first use Young's inequality to obtain
\[\lambda(r) \varphi_\delta(r) + \bar{\lambda}(r) (1-\varphi_\delta(r)) \geq \exp\left( (\log \lambda(r)) \varphi_\delta(r) + (\log \bar{\lambda}(r))(1-\varphi_\delta(r))\right). \]
The right-hand side can be written as
\[\exp \left(  \left( \log \lambda(\delta) + \int_\delta^r \frac{\partial}{\partial t} \log(\lambda(t))\,dt\right) \varphi_{\delta}(r) + \left(\log \bar{\lambda}(\delta) + \int_\delta^r \frac{\partial}{\partial t} \log(\bar{\lambda}(t)) \,dt \right)(1- \varphi_{\delta}(r))   \right).  \]
Since $\bar{\lambda}(r) \geq \lambda(r)$ and $\varphi_\delta$ is increasing, this quantity is greater than 
\begin{align*}
  &  \exp \left(  \log \lambda(\delta) + \left(\int_\delta^r \frac{\partial}{\partial t} \log(\lambda(t))\varphi_\delta(t) \,dt\right) + \left( \int_\delta^r \frac{\partial}{\partial t} \log(\bar{\lambda}(t))(1- \varphi_{\delta}(t)) \,dt\right)   \right) \\
  = & \lambda(\delta) \exp \left(\int_\delta^r g_\delta(t) \,dt \right). 
\end{align*}
This concludes the proof.
\end{proof} 

We now proceed with the proof of \Cref{thm:smooth_approx}.

\begin{proof}[Proof of \Cref{thm:smooth_approx}] 

Let $X$ be a $\CAT(\kappa)$ surface, where $\kappa \geq 0$. Let $(\varepsilon_n)$ be a sequence of positive reals limiting to $0$, and $X_n$ the approximating model polyhedral surface for the value $\varepsilon_n$. In particular, there is a homeomorphic $\varepsilon_n$-isometry $\psi_n \colon X_n \to X$ with the property that $|\psi_n^{-1}(A)|$ converges to $|A|$ for all sets $A \subset X$. For each vertex $v_i$ of $X_n$, $i \in \mathbb{N}$, let $r_i>0$ be such that $B(v_i,2r_i)$ does not contain any other vertex of $X_n$. Note that this implies that the balls $B(v_i,r_i)$ are disjoint. We also require that $\sum_{i=1}^\infty r_i < \varepsilon_n$. The ball $B(v_i,r_i)$ is locally isometric to $B_{\Euc}(\mathbf{0},R_i)$ for some $R_i>0$, equipped with the metric given in \Cref{lemm:polyhedral_approx}. Choose $\delta_i>0$ sufficiently small so that $B_{\Euc}(\mathbf{0},\delta_i)$ has diameter and area at most $\varepsilon_n 2^{-i}$, which can be done since $\lambda_\delta$ is increasing for $r \in (0,\delta)$ and $\lambda(\delta) \to 0$ as $\delta \to 0$. Let $Y_n$ be the smooth Riemannian surface obtained by applying \Cref{lemm:polyhedral_approx} to each neighborhood $B_{\Euc}(\mathbf{0},\delta_i)$. Let $V_i$ denote the neighborhood in $X_n$ on which the metric is altered.

Define a map $\varphi_n$ from $X_n$ to $Y_n$ by mapping $V_i$ to the corresponding neighborhood of $Y_n$ in a continuous way fixing $\partial V_i$ and as the identity map otherwise. By construction, $\varphi_n$ is a $(2\varepsilon_n)$-isometry. Moreover, for any set $A \subset X_n$, we have $|\varphi_n(A)| \leq |A| + \varepsilon_n$. Then $\psi_n \circ \varphi_n^{-1}$ is a $(3\varepsilon)$-isometry from $Y_n$ to $X$ with the property that $|\varphi_n \circ \psi_n^{-1}(A)|$ converges to $|A|$ for all sets $A \subset X$. 
\end{proof}

\begin{rem} \label{rem:cbb_case}
The ideas of \Cref{lemm:polyhedral_approx}  also apply to smoothing vertices of model polyhedral surfaces with curvature bounded below. Assume that $X$ is a model polyhedral surface with curvature bounded below for some $\kappa \leq 0$. Then the metric is given in a neighborhood of a vertex by \eqref{equ:polar_metric}, where $\alpha < 1$. Thus, for sufficiently small $r$, one has $\frac{\partial}{\partial r} \log(\lambda(r))>0$ by \eqref{equ:log_deriv}. Since $\kappa \leq 0$, the chain of inequalities at the end of the first part of the proof of \Cref{lemm:polyhedral_approx} establishes \eqref{equ:curvature_inequality} with ``$\leq$'' replaced by ``$\geq$''. We conclude that the Riemannian approximation in \Cref{lemm:polyhedral_approx} has curvature bounded below by $\kappa$.

The general case can be handled by interpolating with the spherical metric, by modifying the above argument for the hyperbolic metric. We leave the details to the interested reader.  
\end{rem}

\bibliographystyle{abbrv}
\bibliography{bibliography}

\end{document}